\documentclass[hidelinks,onefignum,onetabnum]{siamart220329}

\usepackage{amsfonts}
\usepackage{hyperref}
\usepackage{graphicx}
\usepackage{epstopdf}
\usepackage{algorithm,algpseudocode}
\usepackage{amsmath}
\usepackage{amssymb}
\usepackage{mathrsfs}
\usepackage{bm}
\usepackage[sans]{dsfont}
\usepackage{multirow}
\usepackage{subfigure}
\usepackage{enumerate}
\usepackage{xcolor}
\usepackage{url}
\allowdisplaybreaks

\ifpdf
  \DeclareGraphicsExtensions{.eps,.pdf,.png,.jpg}
\else
  \DeclareGraphicsExtensions{.eps}
\fi

\newsiamremark{remark}{Remark}
\newsiamremark{hypothesis}{Hypothesis}
\crefname{hypothesis}{Hypothesis}{Hypotheses}
\newsiamthm{claim}{Claim}
\newsiamthm{defn}{Definition}

\headers{Fourth-order AMR for incompressible Navier-Stokes}
{Shubo Zhao and Qinghai Zhang}

\title{Fourth-order Adaptive Mesh Refinement
  both in space and in time
  for incompressible Navier-Stokes equations 
  with Dirichlet boundary conditions\thanks
  {
  }
}

\author{Shubo Zhao\thanks
  {
    College of Mathematics and System Sciences,
      Xinjiang University,
      Urumqi, Xinjiang, 830046, China.
    Current address:
      School of Mathematical Sciences,
      Zhejiang University,
      Hangzhou, Zhejiang, 310058, China
      (shubomath@zju.edu.cn).
    This author's work
       was supported by
       the Graduate Student Research Innovation Program (XJ2021G018)
       of the Xinjiang province of China
  }
  \and
  Qinghai Zhang
  \thanks
  { Corresponding author. 
    School of Mathematical Sciences,
      Zhejiang University,
      Hangzhou, Zhejiang, 310058, China
      (\email{qinghai@zju.edu.cn}),
      and College of Mathematics and System Sciences,
      Xinjiang University,
      Urumqi, Xinjiang, 830046, China.
    This author's work
      was supported by
      the Ministry of Science and Technology
      of the People's Republic of China.
  }
}

\usepackage{amsopn}

\ifpdf
\hypersetup{
  pdftitle={Fourth-order adaptive projection methods
    for incompressible Navier-Stokes equations},
  pdfauthor={Shubo Zhao and Qinghai Zhang}
}
\fi

\newcommand{\Dim}{{\scriptsize \textup{D}}}
\newcommand{\ebold}{{\mathbf{e}}}

\newcommand{\ibold}{{\mathbf{i}}}
\newcommand{\jbold}{{\mathbf{j}}}

\newcommand{\iphed}{{\ibold+\frac{1}{2}\ebold^d}}
\newcommand{\iphen}{{\ibold+\frac{1}{2}\ebold^n}}
\newcommand{\ipfed}[1]{\ibold+\frac{#1}{2}\ebold^d}
\newcommand{\imhed}{{\ibold-\frac{1}{2}\ebold^d}}
\newcommand{\imfed}[1]{\ibold-\frac{#1}{2}\ebold^d}
\newcommand{\ipmfed}[1]{\ibold\pm\frac{#1}{2}\ebold^d}

\newcommand{\nbold}{{\mathbf{n}}}
\newcommand{\ubold}{{\mathbf{u}}}

\newcommand{\xbold}{{\mathbf{x}}}
\newcommand{\wbold}{{\mathbf{w}}}

\newcommand{\avg}[1]{\left\langle #1 \right\rangle}
\newcommand{\avgI}[1]{\langle #1 \rangle_{\ibold}}

\newcommand{\nStages}{n_{\tiny \textsf{s}}}
\newcommand{\dt}{k}

\newcommand{\dif}{\,\mathrm{d}}
\newcommand{\Div}{{\mathbf{D}}}
\newcommand{\Grad}{{\mathbf{G}}}
\newcommand{\Iden}{{\mathbf{I}}}
\newcommand{\Lapl}{{\mathbf{L}}}
\renewcommand{\div}{\nabla \cdot}
\newcommand{\grad}{\nabla}
\newcommand{\lapl}{\Delta}

\newcommand{\exOp}{{\mathbf{X}}^{[\text{E}]}}
\newcommand{\Proj}{{\mathbf{P}}}

\newcommand{\cProjLH}{{\mathscr{P}}}

\newcommand{\dom}{\mathcal{D}}
\newcommand{\comp}{\textrm{comp}}
\newcommand{\valid}{\textrm{valid}}
\newcommand{\invalid}{\textrm{invld}}

\begin{document}

\maketitle

\begin{abstract}
  We present a fourth-order projection method
 with adaptive mesh refinement (AMR)
 for numerically solving the incompressible Navier-Stokes equations (INSE)
 with subcycling in time. 
 Our method features
 (i) a reformulation of INSE 
 so that the velocity divergence decays exponentially
 on the coarsest level, 
 (ii) a derivation of coarse-fine interface conditions
 that preserves the decay of velocity divergence
 on any refinement level of the AMR hierarchy,
 (iii) an approximation of the coarse-fine interface conditions
 via spatiotemporal interpolations
 to facilitate subcycling in time, 
 (iv) enforcing to machine precision
 solvability conditions of elliptic equations
 over each connected component of any refinement level, 
 (v) a discrete composite projection
 for synchronizing multiple levels,
 and (vi) geometric multigrid algorithms
 for solving linear systems with optimal complexity. 
Different from current block-structured AMR,
 our method never applies the fine-to-coarse averaging
 to projected velocities.
Results of numerical tests demonstrate
 the high accuracy and efficiency of the proposed method.

 


\end{abstract}



\section{Introduction}
\label{sec:introduction}

Incompressible Navier-Stokes equations (INSE)
 with Dirichlet boundary conditions
 have the dimensionless form
 \begin{subequations}
   \label{eq:INS}
   \begin{align}
     \frac{\partial \ubold}{\partial t}
     + \ubold\cdot\nabla\ubold
     &= \mathbf{g} -\nabla p +  \nu\Delta \ubold
       \quad \text{ in } \dom,
     \\
     \nabla \cdot \ubold &= 0  \qquad\qquad\qquad\quad
                               \text{ in } \dom,
     \\
     \ubold &= \mathbf{u}^{b} \qquad\qquad\qquad
                  \text{ on } \partial\dom,
   \end{align}
 \end{subequations}
 where $t$ is time, 
 $\dom$ a bounded open region in $\mathbb{R}^{\Dim}$,
 $\Dim$ is the dimensionality of the domain, 
 $\partial\dom$ the domain boundary, 
 $\mathbf{g}$ the external force,
 $p$ the pressure,
 $\ubold$ the velocity,
 Re the Reynolds number,
 $\nu:=\frac{1}{\text{Re}}$, 
 and $\mathbf{u}^b$ the Dirichlet boundary condition satisfying
 \begin{equation}
   \label{eq:INSE-solvabilityC}
   \begin{array}{l}
   \int_{\partial\dom}\nbold\cdot\mathbf{u}^b \dif S = 0, 
   \end{array}
 \end{equation}
 which follows from (\ref{eq:INS}b), (\ref{eq:INS}c),
 and the divergence theorem applied to $\dom$.

The INSE in \eqref{eq:INS}
 describe a wide spectrum of real-world phenomena
 such as blood circulation, airflows, and ocean currents.
Numerical simulation of the INSE
 often requires high resolution to capture
 the complex, multiple-scale flow characteristics,
 particularly for medium to high Reynolds numbers.
However,
 uniformly fine grids across the entire domain are
 neither efficient nor necessary,
 as interesting flow features are often localized.
Adaptive mesh refinement (AMR) offers a solution
 to balance efficiency and accuracy
 by refining only local grids of interest.

An AMR method is said to be \emph{synchronized}
 if the local refinement only happens in space
 and all levels share the same time step size;
 see, e.g., those in 
 \mbox{\cite{minion1996projection,howell1997adaptive,zhang2012fourthAMR-ADE,van2022fourth}}.
Instead of a uniform discretization of the entire domain, 
 the spatial grid in the synchronized AMR
 varies in sizes according to the importance of local regions, 
 improving the efficiency
 by reducing the number of spatial unknowns.

An AMR method is \emph{subcycled}
 or have \emph{subcycling in time}
 if local refinements happen not only in space
 but also in time.
Since the time step size of a coarse level 
 is greater than that of the finest level, 
 subcycled AMR is a further improvement
 of synchronized AMR in terms of efficiency.
Osher and Sanders \cite{osher1983numerical}
 proposed a subcycled AMR of the predictor-corrector type 
 to solve the one-dimensional (1D) scalar hyperbolic conservation laws. 
Based on the forward Euler scheme, 
 their method is first-order accurate both in space and in time.
Coquel et al. \cite{coquel2010LTS_IMEX_for_hyperbolic}
 adapted this idea to develop a subcycled AMR
 for 1D hyperbolic systems with moving singularities. 
For hyperbolic equations in two dimensions (2D), 
 Berger and Oliger \cite{berger1984adaptive}
 developed the block-structure AMR 
 where the subcycling consists of
 (i) advancing a coarse level for a single time step, 
 (ii) marching finer levels for multiple time steps
 with boundary conditions approximated from results on the coarse level,
 and (iii) synchronizing the coarse and its finer levels
 once they reach the same time.
This subcycled AMR was applied to gas dynamics
 by Berger and Collella \cite{Berger1989Localadaptivemesh}
 and extended to three dimensions (3D) by
 Bell et al. \cite{Bell1994ThreeDimensionalAdaptiveMesh}.
Almgren et al. \cite{Almgren1998ConservativeAdaptiveProjection}
 proposed a second-order adaptive projection method
 for solving INSE in 2D, 
 which was later extended to 3D by Martin et al.
 \cite{Martin2008cellcenteredadaptiveprojection}.
Results in these works
 demonstrate the superior efficiency of subcycled AMR.
As for accuracy, 
 however, these methods are limited to the second order.
The superiority of subcycled AMR has also been well demonstrated
 in other fields.
Grote et al. \cite{Grote2015RK_ExpLTS_for_wave} developed 
 a series of explicit Runge-Kutta methods with local time-stepping
 for time-dependent simulations of wave propagation,
 effectively overcoming the computational bottleneck
 caused by geometric stiffness.
Recently,
 Berger and LeVeque \cite{berger2024ImpAMR_for_Dispersive}
 proposed a patch-based implicit subcycled AMR method
 for realistic tsunami modeling problems,
 achieving optimal efficiency and minimal numerical dissipation.
 
Under the aforementioned block-structured framework, 
 subcycled AMR has been improved to fourth-order accuracy
 for hyperbolic conservation laws \cite{mccorquodale2011high}
 and recently for compressible Navier-Stokes equations
 \cite{emmett2019fourth,christopher2022high}. 
Two key components of both second-order and fourth-order subcycled AMR 
 are the refluxing at the coarse-fine interface
 and the averaging of values on a fine level
 to replace those on the coarser level; 
 both serve to maintain the consistency of data across multiple levels.

To the best of our knowledge,
 no fourth-order subcycled AMR exists
 for solving the INSE with no-slip or Dirichlet boundary conditions.
The fundamentally different nature of INSE
 from hyperbolic conservation laws
 and compressible flows makes it difficult to
 borrow ideas from successes of current fourth-order subcycled AMR
 \cite{mccorquodale2011high,emmett2019fourth,christopher2022high}.
On the other hand, a straightforward generalization
 of second-order subcycled AMR for INSE
 \cite{Almgren1998ConservativeAdaptiveProjection,Martin2008cellcenteredadaptiveprojection}
 to fourth-order accuracy would lead to order reductions
 and numerical instabilities.
Most of the difficulties concern the solenoidal condition.
\begin{enumerate}[(A)]
\item How to fulfill the divergence-free constraint in (\ref{eq:INS}b)?
  Indeed, a fast increase of velocity divergence, 
  particularly at domain corners, 
  is almost always the precursor to fatal numerical instability.
\item Let $\dom^{\ell}$ denote a subset of $\dom$
    to be refined at the $\ell$th AMR level; 
    see \mbox{Figure \ref{fig:AMRHier}(b)}.
    Then (\ref{eq:INS}b)
    and the divergence theorem over $\dom^{\ell}$
    dictate a compatibility condition
    $\int_{\partial \dom^{\ell}}
    \mathbf{n}\cdot \mathbf{u}|_{\partial \dom^{\ell}} = 0$,
    which, similar to (\ref{eq:INSE-solvabilityC}), 
    should be satisfied as accurately as possible
    to minimize numerial instability, 
    as it is also part of the solvablility conditions of
    pressure Poisson equations (PPEs). 
  In synchronized AMR,
  a fine level
  is embedded in its coarser level at any time, 
  hence it is sufficient to enforce
  $\int_{\partial \dom^{\ell}} \mathbf{n}\cdot \mathbf{u}^b = 0$
  by refluxing at the coarse-fine interface.
  However, in subcycled AMR, 
  the advancing of a fine level inevitably arrives at some time instances 
  when there exist, for the coarse-fine interface,
  neither underlying coarse data nor obvious boundary conditions.
  Then, how do we determine, both theoretically and computationally, 
  these interface conditions so that
  $\int_{\partial \dom^{\ell}} \mathbf{n}\cdot \mathbf{u}^b = 0$
  is satisfied to machine precision 
  without deteriorating the fourth-order accuracy?
\item After a fine level is advanced
  to the time of its coarser level,
  data on the two levels have to be synchronized
  to enforce consistency across the AMR hierarchy.
  In the context of INSE,
  what does this consistency mean?
  How should we synchronize multiple levels
  to prevent numerical instabilities and order reductions?
\end{enumerate}

In this work, 
 we resolve all the above core difficulties
 to propose a fourth-order adaptive projection method
 for solving INSE with subcycling in time.
The governing equations of INSE
 are equivalently reformulated as GePUP-E in \Cref{def:GePUP-E}
 with three variables $\mathbf{w}$, $\mathbf{u}$, and $q$, 
 where the evolutionary velocity $\mathbf{w}$
 needs not be divergence-free
 and the solenoidal velocity $\mathbf{u}$
 and the auxiliary scalar $q$
 are considered as \emph{instantaneous} functions of $\mathbf{w}$
 at any given time instance, cf. (\ref{eq:GePUP-E}c--f). 
As stated in \Cref{thm:decayOfVelDiv-GePUP-E},
 $\nabla\cdot\mathbf{w}$ is governed by a heat equation
 with homogeneous Dirichlet condition
 so that the maximum principle of heat equations
 dictates an exponential decay of $\nabla\cdot\mathbf{w}$,
 indirectly enforcing the solenoidal condition
 and significantly contributing to numerical stability.
This design choice answers (A) from the angle 
 of partial differential equations (PDEs). 

In \Cref{sec:gepup-e-level},
 we further adapt GePUP-E for the subdomain $\dom^{\ell}$
 of the $\ell$th refinement level
 and select the interface condition of $q$ in a way so that
 the exponential decay of $\nabla\cdot\mathbf{w}$
 is preserved over $\dom^{\ell}$.
The GePUP-E formulation in \Cref{def:GePUP-E-OneLevel}
 and the interface conditions in (\ref{eq:GePUP-localized})
 constitute our answer to (B) on the theoretical side.

\Cref{sec:spat-discr} and \Cref{sec:IMEX}
 are brief summaries of
 the finite-volume-based spatial operators
 and the implicit-explicit time integrators, respectively.
The novelty of this work starts
   in \Cref{sec:adapt-mesh-refin} onwards,
   where standard components of block-structured AMR
   are assembled into a synchronized AMR
   which, besides its own algorithmic values,
   also serves to jump start the subcycled AMR
   in \Cref{alg:time-stepping_procedure}.

The interface conditions in (\ref{eq:GePUP-localized})
 are approximated to fourth-order accuracy
 by spatiotemporal interpolations
 detailed in \Cref{subsec:compute-coarse-fine},
 with the compatibility condition 
 $\int_{\partial \dom^{\ell}} \mathbf{n}\cdot \mathbf{u}^b = 0$
 enforced to machine precision
 by a simple device in \Cref{sec:enforc-comp-cond}. 
These algorithms 
 constitute our answer to (B) on the computational side.

The AMR hierarchy of multiple levels
 naturally leads to the concept of composite data
 in (\ref{eq:compositeVariable}),
 which further give rise to a class of composite operators 
 defined by the corresponding single-level operators
 in \Cref{sec:spat-discr}
 and the steps (COH-1,2,3,4) in \Cref{sec:AMR_composite-operators}.
A different (yet crucial) class 
 is the composite projection in (\ref{eq:comp_ProjOp})
 that approximates the Leray-Helmholtz projection $\cProjLH$.
In second-order subcycled AMR methods for INSE
 \cite{howell1997adaptive,Martin2008cellcenteredadaptiveprojection}, 
 the application of a composite projection to the composite velocity
 is often followed by fine-to-coarse averaging. 
However,
 for the fourth-order subcycled AMR,
 this fine-to-coarse averaging results in 
 dominant errors near the coarse-fine interface
 and a reduction of the velocity accuracy to the third order.
Fortunately,
 the fourth-order accuracy is recovered if,
 in synchronizing multiple levels
 (at line 11 in \Cref{alg:time-stepping_procedure}), 
 only the composite projection in (\ref{eq:comp_ProjOp})
 is applied
 without fine-to-coarse averaging of the projected velocity. 
This resolves (C).
 
In \Cref{sec:numerical-tests},
 we perform various benchmark tests 
 to demonstrate the fourth-order accuracy
 and superb efficiency of the proposed method.
In \Cref{sec:conclusion}, 
 we conclude this paper with future research prospects.



\section{The GePUP-E formulation of INSE}
\label{sec:math-form}

In this section,
 we briefly review the GePUP-E formulation \cite{Li2025GePUP-E},
 whose governing equations are
 theoretically equivalent to those of INSE
 but are more amenable to fulfilling the divergence-free condition
 and more conducive to flexible designs of high-order methods.
 
Zhang \cite{zhang2016gepup} proposed the GePUP formulation 
 in which the main evolutionary variable
 is designed to be a velocity $\mathbf{w}$
 that may or may not be solenoidal.
Then the electric boundary conditions
 \cite{shirokoff2011efficient,rosales2021high}
 was adapted into GePUP to form
 the GePUP-E formulation of the INSE (\ref{eq:INS})
 with no-slip conditions \cite{Li2025GePUP-E}, 
 which we slightly generalize in this work to have
 

 
 
\begin{defn}
  \label{def:GePUP-E}
  The \emph{GePUP-E formulation of the INSE (\ref{eq:INS})
    with Dirichlet conditions} $\mathbf{u}^b$ is 
  \begin{subequations}
    \label{eq:GePUP-E}
    \begin{alignat}{2}
      \frac{\partial \wbold}{\partial t} &= \mathbf{g}-\ubold\cdot
      \nabla \ubold-\nabla q+\nu\lapl \wbold &\quad& \ \mathrm{in}\ \dom, \\
      \wbold\cdot \boldsymbol{\tau} &= \mathbf{u}^b\cdot\boldsymbol{\tau}, \quad \nabla \cdot \mathbf{w} = 0
      && \ \mathrm{on}\  \partial \dom, \\
      \ubold &= \cProjLH\wbold && \ \mathrm{in}\  \dom, \\
      \ubold \cdot \nbold &= \mathbf{u}^b\cdot \mathbf{n} && \ \mathrm{on}\  \partial\dom, \\
      \lapl q &= \nabla\cdot(\mathbf{g}-\ubold\cdot\nabla \ubold) && \ \mathrm{in}\  \dom, \\
      \nbold\cdot \nabla q &= \nbold\cdot\left(\mathbf{g} -\ubold\cdot\nabla \ubold
        + \nu\lapl \wbold - \frac{\partial\mathbf{u}^b}{\partial t}\right)
      + \lambda \nbold\cdot(\mathbf{w}-\mathbf{u}^b) && \ \mathrm{on}\  \partial \dom,
    \end{alignat}
  \end{subequations}
where 
  $\mathbf{u}$ is the divergence-free velocity in (\ref{eq:INS}),
  $\mathbf{w}=\mathbf{u}-\nabla \phi$ a non-solenoidal velocity
  for some scalar function $\phi$,
  $\cProjLH$ the Leray-Helmholtz projection,
  $\lambda$ a nonnegative penalty parameter,  
  $\mathbf{n}$ and $\boldsymbol{\tau}$
  the unit normal and unit tangent vectors
  of $\partial \dom$, respectively.
 The two velocities have the same initial condition
  in $\overline{\cal D}$, the closure of ${\cal D}$, 
  i.e.,
  \begin{equation}
    \label{eq:initialConditionW}
   \forall \xbold \in \overline{\dom}, \quad
   \mathbf{w}(\xbold,t_0) = \mathbf{u}(\xbold, t_0).
 \end{equation}
\end{defn}

By the arguments in \cite[Section 3]{Li2025GePUP-E}, 
 it is straightforward to prove
 the equivalence of the INSE in (\ref{eq:INS})
 and the GePUP-E in (\ref{eq:GePUP-E}). 
Then the equivalence of $\mathbf{w}$ and $\mathbf{u}$
 and the boundary condition (\ref{eq:INS}c) yield
 $\mathbf{n}\cdot\mathbf{w} = \mathbf{n}\cdot\mathbf{u}^b$, 
 which, however, 
 is not explicitly included in (\ref{eq:GePUP-E}). 
Instead, the normal component of (\ref{eq:GePUP-E}a) and 
 (\ref{eq:GePUP-E}f) imply
 \begin{equation}
   \label{eq:normalVel}
   \frac{\partial }{\partial t}
   \left[\mathbf{n}\cdot(\mathbf{w}-\mathbf{u}^b)\right]
   =
   -\lambda\mathbf{n}\cdot(\mathbf{w}-\mathbf{u}^b) \quad \mathrm{on}\ \partial\dom,
 \end{equation}
 which, together with (\ref{eq:initialConditionW}) and (\ref{eq:GePUP-E}d), 
 gives 
 $\mathbf{n} \cdot (\mathbf{w}(t) - \mathbf{u}^b(t))=0$
 for any $t\ge t_0$.
In addition, 
 (\ref{eq:GePUP-E}b) implies the following Neumann condition
 for $w_n:=\mathbf{w}\cdot\mathbf{n}$ in
 (\ref{eq:GePUP-E}a):
 \begin{equation}
   \label{eq:NeumannConditionForNormalVel}
   \frac{\partial w_n}{\partial n} =
   -\sum\nolimits^{\Dim-1}_{i=1}
   \frac{\partial u^b_{\tau_i}}{\partial \tau_i}
   \quad \mathrm{on}~ \partial \dom,
 \end{equation}
 where the subscript ``$_{\tau_i}$''
 denotes the $i$th tangential component.
 

The most important feature of GePUP-E
 is the exponential decay of $\nabla\cdot\mathbf{w}$. 
 
 \begin{theorem}
  \label{thm:decayOfVelDiv-GePUP-E}
  The evolution of $\nabla\cdot\mathbf{w}$
  in the GePUP-E formulation (\ref{eq:GePUP-E})
  is governed by the heat equation
  with homogeneous Dirichlet conditions, 
   \begin{equation}
     \label{eq:heatEqOfDivW}
     \left\{
       \begin{array}{rll}
         \frac{\partial\left(\nabla\cdot \mathbf{w}\right)}{\partial t}
         &= \nu\Delta\left(\nabla\cdot \mathbf{w}\right)
         &\quad \mathrm{in}\ \dom, \\
         \nabla\cdot \mathbf{w} &= 0 &\quad \mathrm{on}\ \partial\dom,
       \end{array}
     \right.
   \end{equation}
   which implies 
   $\|\nabla\cdot\mathbf{w}(t)\| \le
     e^{-\nu C(t-t_0)} \|\nabla\cdot\mathbf{w}(t_0)\|$
   where $t_0$ is the initial time
   and $C$ a positive constant independent of $\mathbf{w}$.
 \end{theorem}
 \begin{proof}
   See \cite[Theorem 4]{Li2025GePUP-E}.
 \end{proof}

Theorem \ref{thm:decayOfVelDiv-GePUP-E}
 and the initial condition (\ref{eq:initialConditionW})
 dictate that $\nabla\cdot \mathbf{w}=0$ always holds.
In practical computations, however, 
$\Div \avg{\mathbf{w}}$,
 the discrete counterpart of
 $\nabla\cdot \mathbf{w}$, might not be zero
 at each time step, 
 due to the imperfection of the discrete projection  
 and the discretization errors of other spatial operators
 such as $\mathbf{u}\cdot\nabla \mathbf{u}$.
Then the exponential decay of $\nabla\cdot \mathbf{w}$
 becomes valuable 
 in controlling the evolution of $\Div \avg{\mathbf{w}}$.
Such a mechanism is fully exploited in this work.
 



\section{Spatial discretization based on finite volumes}
\label{sec:spat-discr}



Hereafter we assume that the problem domain $\dom$ is a rectangle
 so that it can be partitioned into
 a structured array of \emph{control volume}s or \emph{cell}s, 
 each of which is a square $\mathcal{C}_{\ibold}$ of size $h$
 with the multi-index $\ibold \in \mathbb{Z}^{\Dim}$ 
 indicating the rectangular structure of the grids.
The high and low faces of a cell $\mathcal{C}_{\ibold}$
 along dimension $d$ are denoted
 by $\mathcal{F}_{\ipfed{1}}$ and $\mathcal{F}_{\imfed{1}}$,
 respectively. 

For a function $\phi: \mathbb{R}^{\Dim}\rightarrow \mathbb{R}$, 
its \emph{cell-averaged} value over $\mathcal{C}_{\ibold}$
and its \emph{face-averaged} value over $\mathcal{F}_{\ipmfed{1}}$
are respectively given by
\begin{equation}
  \label{eq:averagedValues}
  \avgI{\phi}
  := \frac{1}{h^{\Dim}}\int_{{\mathcal C}_{\ibold}}
  \phi\left(\mathbf{x}\right)\,\dif \mathbf{x};
  \quad
   \avg{\phi}_{\ipmfed{1}}
   := \frac{1}{h^{\Dim-1}}
   \int_{{\mathcal F}_{\ipmfed{1}}}
   \phi\left(\mathbf{x}\right)\,\dif \mathbf{x}.
\end{equation}

As shown in \cite[Appendix A]{zhang2012fourthAMR-ADE}, 
 cell averages can be converted to face averages by
 \begin{equation}
   \label{eq:cell2face}
   \begin{array}{rl}
     \avg{\phi}_{\ipfed{1}}
     &= 
     \frac{1}{12}\left(
     - \avg{\phi}_{\ibold+2\ebold^d} 
     + 7 \avg{\phi}_{\ibold+\ebold^{d}} + 7 \avgI{\phi} 
     - \avg{\phi}_{\ibold-\ebold^d}
     \right) + O(h^4),\\
     \avg{\frac{\partial \phi}{\partial x_d}}_{\ipfed{1}}
     &= 
     \frac{1}{12 h}\left(
     - \avg{\phi}_{\ibold+2\ebold^d} 
     + 15 \avg{\phi}_{\ibold+\ebold^{d}} - 15 \avgI{\phi} 
     + \avg{\phi}_{\ibold-\ebold^d}
     \right) + O(h^4).
   \end{array}
 \end{equation}

The discrete gradient, the discrete divergence,
 and the discrete Laplacian
 act on cell averages as follows. 
 \begin{subequations}
   \label{eq:discreteOps}
   \begin{align}
     \Grad_d \avg{\phi}_{\ibold}
     &:= \frac{1}{12 h} \left(
     - \avg{\phi}_{\ibold+2\ebold^d} 
     + 8 \avg{\phi}_{\ibold+\ebold^d} - 8 \avg{\phi}_{\ibold-\ebold^d}
     + \avg{\phi}_{\ibold-2\ebold^d}
     \right),
     \\
     \Div \avg{\ubold}_{\ibold}
       &:= \frac{1}{12 h} \sum_d\left( 
         - \avg{u_d}_{\ibold+2\ebold^d}
         + 8 \avg{u_d}_{\ibold+\ebold^d} - 8 \avg{u_d}_{\ibold-\ebold^d} 
         + \avg{u_d}_{\ibold-2\ebold^d}
         \right),
     \\
     \Lapl \avg{\phi}_{\ibold}
       &:= \frac{1}{12h^2}
         \sum_{d} \left(
         \!-\! \avg{\phi}_{\ibold+ 2\ebold^{d}}
         \!+\! 16 \avg{\phi}_{\ibold+\ebold^{d}} 
         \!-\! 30 \avg{\phi}_{\ibold}
         \!+\! 16 \avg{\phi}_{\ibold-\ebold^{d}} 
         \!-\! \avg{\phi}_{\ibold - 2\ebold^{d}}
         \right).
       \end{align}
 \end{subequations}
 
In particular,
 the cell-averaged convection can be approximated by
 \begin{equation}
   \label{eq:convOp}
   \Div \avg{\mathbf{u}\mathbf{u}}_{\ibold}
   := \frac{1}{h} \sum_d 
   \left(
     \mathbf{F}\avg{u_d,\mathbf{u}}_{\iphed} 
     -\mathbf{F}\avg{u_d,\mathbf{u}}_{\imhed} 
   \right),
 \end{equation}
 where the face average of the product of two scalars is given by 
\begin{displaymath}
  \begin{array}{l}
  \mathbf{F}\avg{\varphi,\psi}_{\ipfed{1}}
  := 
  \avg{\varphi}_{\ipfed{1}}  \avg{\psi}_{\ipfed{1}} 
  +\frac{h^2}{12}
  \sum \limits_{d'\ne d}
  \left(\Grad_{d'}^{\perp}\varphi \right)_{\ipfed{1}}
  \left( \Grad_{d'}^{\perp}\psi \right)_{\ipfed{1}},
  \end{array}
\end{displaymath}
 and
 $\Grad_{d'}^{\perp}$ is the discrete gradient
 in the transverse directions
 \begin{displaymath}
   \left( \Grad_{d'}^{\perp} \varphi \right)_{\ipfed{1}} 
   :=  \frac{1}{2h}\left(
     \avg{\varphi}_{\ipfed{1}+\ebold^{d'}}
     - \avg{\varphi}_{\ipfed{1}-\ebold^{d'}}
   \right).
 \end{displaymath}
 
 It is shown in
 \cite[Proposition 1]{Zhang2014fourthorderapproximateprojection}   
 that $\Grad$, $\Div$, $\Lapl$,
 and $\Div\avg{\mathbf{u}\mathbf{u}}$
 are all fourth-order accurate in
 approximating
 cell averages of $\nabla$, $\nabla\cdot$, $\Delta$,
 and $\mathbf{u}\cdot\nabla\mathbf{u}$, respectively.

A discrete approximate projection
 is formed from the above discrete operators, 
 \begin{equation}
   \label{eq:appProj}
   \Proj := \Iden - \Grad\Lapl^{-1}\Div,
 \end{equation}
 which is fourth-order accurate
 in approximating $\cProjLH$ in (\ref{eq:GePUP-E}c)
 \cite[Theorem 5]{Zhang2014fourthorderapproximateprojection}.
The numerical stability of $\Proj$
 on no-penetration rectangular domains was demonstrated in
 \cite{zhang2016gepup}
 by verifying its spectral radius being one
 within machine precision on uniform grids; 
 see \cite[Section 4]{zhang2016gepup}
 for more details on the implementation of $\Proj$.

To facilitate the evaluation of discrete operators
 in (\ref{eq:discreteOps}), 
 we wrap up the rectangular grid
 with two layers of ghost cells
 and set values of these ghost cells
 according to different boundary conditions 
 so that each formula in (\ref{eq:discreteOps})
 stays the same for all cell averages within the domain, 
 even if the cell is abutting the domain boundary. 
This decouples discretizing spatial operators
 from fulfilling physical boundary conditions. 
 
For non-periodic domain boundaries,
 values of ghost cells are calculated to the fifth-order accuracy
 from interior cell averages and the given boundary conditions.
For example,
 Dirichlet boundary conditions 
 are fulfilled
 by filling the ghost cells with
 \begin{equation}
   \label{eq:velBCHomoDiri5th}
   {
     \footnotesize
     \begin{aligned}
       \avg{\phi}_{\ibold+ \ebold^n}
       \!=& \frac{1}{12}
       \!\left(
         -77 \avg{\phi}_{\ibold}\! +\! 43 \avg{\phi}_{\ibold- \ebold^n}
         \!-\!17 \avg{\phi}_{\ibold-2 \ebold^n} \!+\! 3 \avg{\phi}_{\ibold-3\ebold^n}
       \right)
       \!+\! 5\avg{\phi}_{\ibold+\frac{1}{2}\ebold^n}  
       \!+\! O(h^5),\\
       \avg{\phi}_{\ibold+ 2\ebold^n}
       \!=& \frac{1}{12}
       \!\left(
         -505 \avg{\phi}_{\ibold} \!+\! 335 \avg{\phi}_{\ibold- \ebold^n}
         \!-\!145 \avg{\phi}_{\ibold-2 \ebold^n} \!+\! 27 \avg{\phi}_{\ibold-3\ebold^n}
       \right)
       \!+\! \frac{75}{3}\avg{\phi}_{\ibold+\frac{1}{2}\ebold^n}
       \!+\! O(h^5),
     \end{aligned}
     }
   \end{equation}
   where ``$\ebold^n$'' denotes the unit normal vector
   of a local arc of the domain boundary,
   ${\cal F}_{\ibold+\frac{1}{2}\ebold^n}\subset \partial {\cal D}$
   the high face of the interior cell ${\cal C}_{\ibold}$,
   and $\avg{\phi}_{\ibold+\frac{1}{2}\ebold^n}$
   the face-averaged Dirichlet condition. 
Similarly, 
a Neumann boundary condition is fulfilled by setting
 \begin{equation}
   \label{eq:Neumann5}
   {
     \footnotesize
     \begin{aligned}
       \avg{\phi}_{\ibold+ \ebold^n}
       \!=& \frac{1}{10}
       \!\left(
         5 \avg{\phi}_{\ibold} \!+\! 9 \avg{\phi}_{\ibold- \ebold^n}
         \!-5 \avg{\phi}_{\ibold-2 \ebold^n} \!+\!  \avg{\phi}_{\ibold-3 \ebold^n}
       \right) 
       \!+\! \frac{6h}{5}\avg{\frac{\partial \phi}{\partial n}}_\iphen
       \!+\! O(h^5),
       \\
       \avg{\phi}_{\ibold+ 2\ebold^n}
       \!= &\frac{1}{10}
       \!\left(
         -75\! \avg{\phi}_{\ibold} \!+\! 145\! \avg{\phi}_{\ibold- \ebold^n}
         \!-\!75 \avg{\phi}_{\ibold-2 \ebold^n} \!+\! 15 \avg{\phi}_{\ibold-3 \ebold^n} 
       \right)
       \!+\! 6 h \avg{\frac{\partial \phi}{\partial n}}_\iphen  
       \!+\! O(h^5),
     \end{aligned}}
 \end{equation}
 where $\avg{\frac{\partial \phi}{\partial n}}_\iphen$ denotes
 the face-averaged Neumann condition.


Sometimes we need to approximate 
 the face average of a normal derivative 
 from a Dirichlet boundary condition and interior cell averages; 
 this is done by calculating 
 \begin{equation}
   \label{eq:smoothExtend5p-face-deri}
   {\scriptsize
     \begin{aligned}
       \avg{\frac{\partial \phi}{\partial n}}_{\iphen}
       \!=\!\frac{1}{72h}
       \!\left(
         -415\! \avg{\phi}_{\ibold} \!+\! 161\! \avg{\phi}_{\ibold-\ebold^n}
         \!-\! 55 \! \avg{\phi}_{\ibold-2\ebold^n} 
         \!+\! 9 \! \avg{\phi}_{\ibold-3\ebold^n}
       \right) 
       \!+\! \frac{25}{6h}\avg{\phi}_{\iphen}
       \!+\! O(h^4).
     \end{aligned}
   }
 \end{equation}
 
Thanks to the fifth-order accuracy
 of (\ref{eq:velBCHomoDiri5th}) and (\ref{eq:Neumann5}),
 operators of first derivatives
 are discretized to the fourth order.
 (\ref{eq:smoothExtend5p-face-deri})
 is aligned with this spirit. 
As for second derivatives such as the Laplacian,
 (\ref{eq:velBCHomoDiri5th}) and (\ref{eq:Neumann5})
 lead to third-order truncation errors
 on a set of codimension one
 near the domain boundary.

 


\section{Time integration via implicit-explicit schemes}
\label{sec:IMEX}

Integrate (\ref{eq:GePUP-E}a) over the control volumes,
 apply the definition of the discrete operators in
 (\ref{eq:discreteOps})
 and (\ref{eq:convOp}), 
 neglect the truncation errors, 
 and we obtain a system of ODEs, 
\begin{equation}
  \label{eq:ODE-MOL}
  \frac{\dif \avg{\mathbf{w}}}{\dif t}
  = \exOp + \nu\Lapl\avg{\mathbf{w}},
\end{equation}
where the diffusion term $\nu\Lapl\avg{\mathbf{w}}$
 is stiff for large $\nu$
 while $\exOp := \avg{\mathbf{g}}
 -\Div \avg{\mathbf{u}\mathbf{u}} - \Grad\avg{q}$
 is not stiff. 
By (\ref{eq:initialConditionW}),
 we supplement (\ref{eq:ODE-MOL}) with the initial condition
 $\avg{\mathbf{w}}(t_0) = \avg{\mathbf{u}}(t_0)$.

As a prominent feature of the GePUP-E formulation, 
 $\mathbf{u}$ is determined instantanously
 from the main evolutionary variable $\mathbf{w}$
 and $q$ is also determined instantaneously
 from $\mathbf{w}$ and $\mathbf{u}$.
As such, 
a time integrator can be employed \emph{in a black-box manner}
to solve (\ref{eq:ODE-MOL}). 
For example,
 an explicit Runge-Kutta method is a good choice
 for flows with very high Reynolds numbers.
On the other hand,
 for flows of low/medium Reynolds numbers,
 we would like to switch to an additive Runge-Kutta method
 so that the stiff and nonstiff terms can be treated
 with implicit and explicit Runge-Kutta methods, respectively.
The agility of switching the time integrator
 without worrying about its internal details
 is enabled by the fact of $\mathbf{w}$
 being the only evolutionary variable
 without the divergence-free constraint; 
 see \cite[Section 1]{Li2025GePUP-E}
 for more discussions.
 
In this work, we choose the time integrator to be
 the ERK-ESDIRK method \cite{kennedy2003additive},
 an implicit-explicit (IMEX) scheme in the family of
 additive Runge-Kutta methods.
The detailed steps for solving (\ref{eq:ODE-MOL})
are as follows. 
 \begin{subequations}
   \label{eq:GePUP-IMEX}
   \renewcommand{\arraystretch}{1.2}
   \begin{equation}
     \avg{\mathbf{w}}^{(1)} = \avg{\mathbf{u}}^{(1)} = \avg{\mathbf{w}}^n,
   \end{equation}
   \vspace{-20pt}
   \begin{equation}
     \left\{
     \begin{array}{rl}
       \text{for } s = 2, 3, \ldots, \nStages,& 
       \\
       \Lapl \avg{q}^{(s-1)} =& 
       \Div\bigl( 
       \avg{\mathbf{g}}^{(s-1)} - \Div\avg{\mathbf{u}\mathbf{u}}^{(s-1)} 
       \bigr)
       \\
       \bigl(\Iden - \dt\nu\gamma \Lapl\bigr) \avg{\mathbf{w}}^{(s)}
     =&
     \, \avg{\mathbf{w}}^n
     +\dt \sum^{s-1}_{j=1}  a^{[\text{E}]}_{s,j} \exOp
     \left(\avg{\mathbf{u}}^{(j)}, \avg{q}^{(j)}, t^{(j)}\right)
     \\
     &+ \dt\nu \sum^{s-1}_{j=1}  a^{[\text{I}]}_{s,j} \Lapl\avg{\mathbf{w}}^{(j)},
     \\
     \avg{\mathbf{u}}^{(s)} =& \Proj \avg{\mathbf{w}}^{(s)},
     \end{array}
     \right.
   \end{equation}
   \vspace{-10pt}
   \begin{equation}
     \left\{
     \begin{array}{rl}
     \avg{\mathbf{w}}^{*} =& \avg{\mathbf{w}}^{(\nStages)}
     + \dt \sum^{\nStages}_{j=1} \left(b_j - a^{[E]}_{\nStages,j}\right) 
       \exOp \left(\avg{\mathbf{u}}^{(j)}, \avg{q}^{(j)}, t^{(j)}\right),
     \\
      \avg{\mathbf{u}}^{n+1} =& \Proj \avg{\mathbf{w}}^{*},
      \\
      \avg{\mathbf{w}}^{n+1} =& \avg{\mathbf{u}}^{n+1},
     \end{array}
     \right.
   \end{equation}
  \end{subequations}
  where the superscript ``$^{(s)}$'' denotes the $s$th intermediate stage,
  $t^{(s)} = t_n + c_{s} \dt $ the time of that stage,
  $a^{[E]}_{s,j}$, $a^{[I]}_{s,j}$, $b_j$, $c_s$ 
  the standard coefficients of the \emph{Butcher tableau} in
  \cite[Appendix C]{zhang2012fourthAMR-ADE},
  and $\gamma = a^{[I]}_{s,s}$.
  The boundary conditions for
  the projection step in (\ref{eq:GePUP-IMEX}b,c)
  and the first linear system in (\ref{eq:GePUP-IMEX}b)
  are the face averages of (\ref{eq:GePUP-E}d) and (\ref{eq:GePUP-E}f), respectively.
  As for the second linear system in (\ref{eq:GePUP-IMEX}b), 
  the boundary conditions for the normal and tangential components of $\avg{\mathbf{w}}$
  are face averages of
  (\ref{eq:NeumannConditionForNormalVel})
  and $\mathbf{w}\cdot\boldsymbol{\tau} =
  \mathbf{u}^b\cdot\boldsymbol{\tau}$
  in (\ref{eq:GePUP-E}b), respectively.



\section{Patch-based local refinement}
\label{sec:adapt-mesh-refin}

After defining the concept of an AMR hierarchy
 in \Cref{sec:amr-hierarchy}, 
 we describe in \Cref{sec:data-transf-betw}
 data transfers between adjacent levels
 and extend in \Cref{sec:AMR_composite-operators}
 the single-level discrete operators in (\ref{eq:discreteOps})
 to composite operators
 that act on multiple adjacent levels.
\Cref{subsec:regridding} concerns
 data migration from one AMR hierarchy to another.
Finally, 
 a synchronized AMR method for INSE 
 is introduced in \Cref{sec:an-amr-method-nonsubcycling} 
 as a natural consequence of
 combining the IMEX scheme (\ref{eq:GePUP-IMEX})
 with these AMR components. 
 
\subsection{The AMR hierarchy}
\label{sec:amr-hierarchy}

Denote by $\Upsilon$ a set of grids with different sizes, 
\begin{equation}
  \label{eq:grids}
    \Upsilon := \left\{\Upsilon^\ell :
    \Upsilon^{\ell} \subset \mathbb{Z}^D, \ 
    \Upsilon^{\ell+1} = \mathcal{C}^{-1}_{r}\Upsilon^{\ell},\  
    \ell = 0,1,\cdots, \ell_{\max}
  \right\},
\end{equation}
where each $\Upsilon^{\ell}$ discretizes the domain $\dom$, 
 the \emph{refinement ratio} $r$ 
 is for two successive discretizations, 
 and the \emph{coarsening operator}
 $\mathcal{C}_r : \mathbb{Z}^{\Dim}\rightarrow\mathbb{Z}^{\Dim}$
 is given by
 ${\mathcal{C}}_r(\ibold) = \left\lfloor  \frac{\ibold}{r} \right\rfloor$.
Then the grid sizes of two adjacent discretizations are related
 by $h^{\ell}=r h^{\ell+1}$.
Typically $r$ is assumed to be a constant,
 either 2 or 4,
 across all levels.

As illustrated in \Cref{fig:AMRHier}, 
 an \emph{AMR hierarchy} $\Omega_{\Upsilon}$
 is a set of consecutive \emph{AMR levels} $\Omega^{\ell}$, 
 each of which consists of a number of 
 pairwise disjoint \emph{box}es or \emph{patch}es $\Omega^{\ell}_k$:
 \begin{equation}
   \label{eq:AMRHier}
   \begin{array}{rl}
     \Omega_{\Upsilon}
     &:= \left\{
       \Omega^\ell :\ \Omega^0=\Upsilon^0;\ 
       \forall \ell>0,\ \Omega^{\ell} \subset \Upsilon^\ell
       \right\},
     \\
     \Omega^\ell
     &:= \left\{ 
       \Omega^{\ell}_k : \;  
       k \neq j\, \Leftrightarrow\, 
       \Omega^{\ell}_k \cap \Omega^{\ell}_{j} = \emptyset
       \right\},
     \\
     \Omega^{\ell}_k(\jbold_{\min}, \jbold_{\max})
     &:= \left\{
       \ibold \in\mathbb{Z}^{\Dim} :\ 
       \jbold_{\min} \le \ibold \le \jbold_{\max}
       \right\},
   \end{array}
 \end{equation}
 where each patch $\Omega^{\ell}_k$ is a rectangular box
 uniquely determined by the two multi-indices $\jbold_{\min}$
 and $\jbold_{\max}$
 and ``$\ibold \le \jbold$'' holds if and only if
 ``$i_d\le j_d$'' holds for each $d=1,\ldots, \Dim$.

 \begin{figure}
   \centering
   \subfigure[the domain $\Omega^0$ with two refinement levels]{
     \includegraphics[scale=0.635]{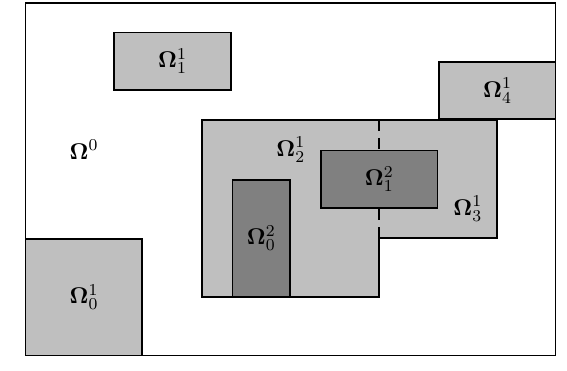}
   }
   \subfigure[subdomain $\dom^1 = \mathop{\cup}\limits^{2}_{c=0}\dom^{1,c} $ and its boundary]{
     \includegraphics[scale=0.635]{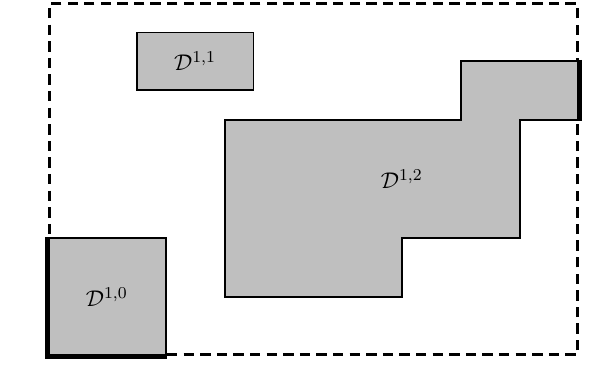}
   }
   \subfigure[filling ghost cells for $\Omega^1_4$]{
     \includegraphics[scale=0.46]{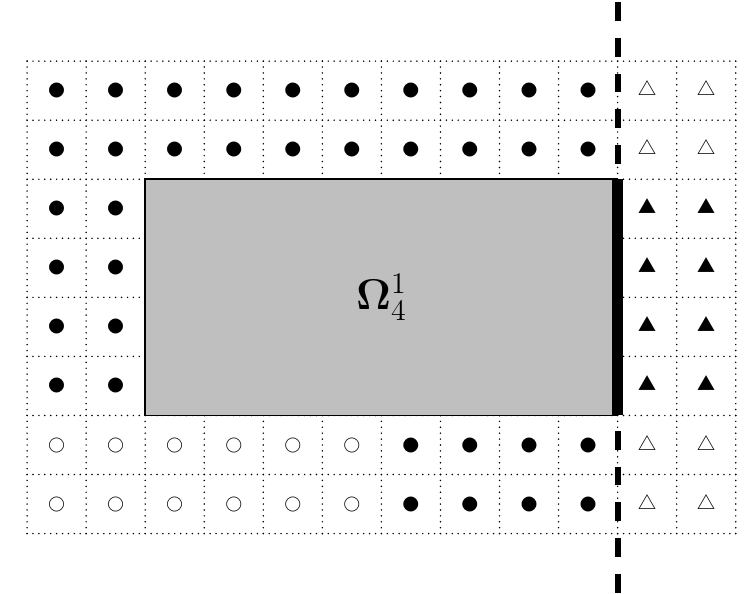}
   }
   \subfigure[the invalid  and valid region of $\Omega^1$]{
     \includegraphics[scale=0.635]{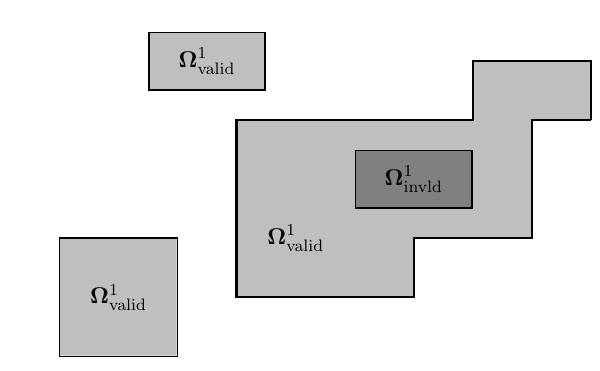}
   }
   \caption{An AMR hierarchy.
     In subplot (a), 
     $\Omega^1_0$ and $\Omega^1_4$ are properly nested
     at the physical boundary; 
     $\Omega^1_1$, $\Omega^1_2$, $\Omega^1_3$,
     and $\Omega^2_1$
     are properly nested at the coarse-fine interface;
     but $\Omega^2_0$ is \emph{not} properly nested.
     In subplot (b),
     $\dom^{1,c}$ denotes the $c$th component of the first subdomain,
     the dashed lines the physical boundaries $\partial\dom$,
     and the thick and thin solid lines
     $\partial\dom^1_P$ and $\partial\dom^1_I$, respectively.
     In subplot (c),
     all ghost cells of patch $\Omega^1_4$ are presented,
     where ``$\circ$'' represents a ghost cell to be filled by
     exchanging values of adjacent patches,
     ``$\bullet$'' a ghost cell to be filled by AMRCFI,
     ``$\vartriangle$'' and ``$\blacktriangle$''
     those to be filled by
     (\ref{eq:velBCHomoDiri5th}) or (\ref{eq:Neumann5}).
     Those ``$\vartriangle$'' ghost cells have to be filled
     after ``$\bullet$''s and ``$\circ$''s are filled.
     In subplot (d),
       the invalid region
       $\Omega^1_{\mathrm{invld}}= \mathcal{C}_{r}(\Omega^2_1)$ is shown in gray
       (the patch $\Omega^2_0$ is not shown due to its improper nesting), 
       whereas patches of the valid region
       $\Omega^1_{\mathrm{valid}} = \Omega^1 \setminus \Omega^1_{\mathrm{invld}}$
       are shaded in light gray.
     }
   \label{fig:AMRHier}
 \end{figure}

\emph{The $\ell$th subdomain} $\mathcal{D}^{\ell}$
 is the region discretized by grids in $\Omega^{\ell}$
 that satisfies 
 \begin{equation}
   \label{eq:embeddingOfSubdomain}
   \forall \ell=1,2,\ldots,\ell_{\max},\quad
   \mathcal{D}^{\ell}\subset \mathcal{D}^{\ell-1},
 \end{equation}
 where $\mathcal{D}^0:= \mathcal{D}$.
As shown in \Cref{fig:AMRHier}(b), 
 $\mathcal{D}^\ell$ may contain multiple connected regions
 with its boundary as 
 \begin{equation}
   \label{eq:subdomainBoundary}
   \partial \dom^{\ell}
   = \partial \dom^{\ell}_{P}\cup \partial \dom^{\ell}_{I}
 \end{equation}
 where \emph{the physical boundary} 
 and \emph{the coarse-fine interface at level $\ell$} are respectively
 \begin{equation}
   \label{eq:subdomainBdryType}
   \partial\dom^{\ell}_P
   := \partial \dom^{\ell} \cap \partial \dom; \quad
   \partial\dom^{\ell}_I
   := \partial \dom^{\ell} \setminus \partial \dom.
 \end{equation}
 
 The area ${\cal C}_{\ibold}$ of the $\ibold$th coarse cell
   is also occupied by the union of fine cells
   that refine ${\cal C}_{\ibold}$.
   Hence the average of a function $\phi$ over ${\cal C}_{\ibold}$ 
   can either be represented by $\avg{\phi}_{\ibold}$
   in (\ref{eq:averagedValues})
   or the average of averages of $\phi$ over the fine cells.
 This potential ambiguity can be resolved
 by defining variables only on the \emph{valid region}
 \begin{equation}
   \label{eq:validRegion}
   \begin{array}{l}
   \Omega_{\valid}^{\ell} :=
   \left(\bigcup_{\Omega^{\ell}_k \in \Omega^{\ell}}
     \Omega^{\ell}_k\right)
     \setminus
     \Omega_{\invalid}^{\ell}, 
   \end{array}
 \end{equation}
 where $\Omega_{\invalid}^{\ell}:=
  \bigcup_{\Omega^{\ell+1}_k \in \Omega^{\ell+1}}
  \mathcal{C}_{r}(\Omega^{\ell+1}_k)$
  is the \emph{invalid region} of $\Omega^{\ell}$.
  In other words,
    the invalid region of the $\ell$th level
    contains and only contains cells
    that are refined in the $(\ell+1)$th level;
    see Figure \ref{fig:AMRHier}(d) for
    some examples of $\Omega_{\invalid}^{1}$ and $\Omega_{\valid}^{1}$.

An AMR hierarchy $\Omega_{\Upsilon}$ must satisfy
 the \emph{proper refinement condition},
\begin{equation}
  \label{eq:pBlock}
  \forall \ell>0,\ \forall \Omega^{\ell}_k\in \Omega^{\ell},
  \quad
  \Omega^{\ell}_k = \mathcal{C}^{-1}_{r}
  \left(\mathcal{C}_{r}(\Omega^{\ell}_k)\right), 
\end{equation}
 and the \emph{proper nesting condition} that, 
 along any direction, 
 there be at least one control volume in $\Omega^{\ell}$
 separating $\Omega_{\valid}^{\ell+1}$ from $\Omega_{\valid}^{\ell-1}$
 to prevent abrupt changes of any functions on $\Omega_{\valid}^{\ell}$; 
 see Figure \ref{fig:AMRHier}(a).

We define the \emph{composite data of a function $\phi$
  over an AMR hierarchy} as
\begin{equation}
  \label{eq:compositeVariable}
  \avg{\phi}^{\comp}_{\ell_{\text{lo}}}
  := \left\{ \avg{\phi}^{\ell}:
    \ell = \ell_{\text{lo}}, \ell_{\text{lo}}+1, \ldots, \ell_{\max}
  \right\}
\end{equation}
 where $\avg{\phi}^{\ell}$ denotes the \emph{data on the $\ell$th level},
 i.e., the array of averaged values of
 $\phi$ over cells in the $\ell$th level $\Omega^{\ell}$.
 In particular, we write
 $\avg{\phi}^{\comp}:=\avg{\phi}^{\comp}_{0}$.

\subsection{Data transfers between adjacent levels}
\label{sec:data-transf-betw}

The only natural application of
 transferring data on a fine level to the adjacent coarse level
 appears to be the
 replacement of the coarse level data by the fine level data, 
 \begin{equation}
   \label{eq:AMR_avergingDownToCrse}
   \forall \ibold \in 
   \mathcal{C}_r(\Omega^{\ell+1}), \quad
   \avg{\phi}^{\ell}_{\ibold}
   = \frac{1}{r^\Dim} \sum\nolimits_{\jbold\in\mathcal{C}^{-1}_r(\{\ibold\})}
   \avg{\phi}^{\ell+1}_{\jbold}.
 \end{equation}
The data of $\Omega^{\ell}$
 on the invalid region $\mathcal{C}_r(\Omega^{\ell+1})$
 is said to be \emph{redundant}
 if (\ref{eq:AMR_avergingDownToCrse}) holds. 

Conversely, 
 data transfer from a coarse level to its finer level is needed for
 \begin{itemize}
 \item initializing cell averages on newly refined grids,
 \item filling ghost cells along the coarse-fine interface $\partial\dom^{\ell}_I$ for each patch.
 \end{itemize}

These scenarios are handled by formulas of AMRCFI
 \cite{zhang11:_amrcf,zhang11:_high_order_multid_coars_fine},
 a family of algorithms for
 efficient, generic, and conservative coarse-fine interpolation (CFI)
 based on multi-dimensional polynomials.
AMRCFI is conservative in that
 the conservation constraint
 \eqref{eq:AMR_avergingDownToCrse}
 is satisfied exactly.
It is also efficient
 in that multi-dimensional polynomial interpolation is reduced
 to multiplying a predetermined matrix
 to the vector of cell averages on the coarse level,
 where the matrix depends only
 on $r$ and the relative positions of the fine cells
 to the coarse cells.
As such, no linear system is solved at the run time
 and the ill-conditioning of Vandermonde matrix is avoided.
The AMRCFI formulas are chosen to be fifth-order accurate
 for reasons similar to those in the last paragraph
 of \Cref{sec:spat-discr}.

The stencils for multi-dimensional polynomial reconstruction
 in AMRCFI
 are generated by a heuristic algorithm,
 whose validity can be rigorously proved
 by the concept of triangular lattices
 recently introduced in \cite{zhang2024PLG}.

\subsection{Composite operators over a hierarchy (COH)}
\label{sec:AMR_composite-operators}

Any discrete operator $\mathbf{Q}$
 that acts on level data, 
 such as those in \Cref{sec:spat-discr}, 
 can now be generalized to a \emph{composite operator}
 $\mathbf{Q}_{\ell_{\text{lo}}}^{\comp}$ 
 that acts on composite data
 $\avg{\phi}_{\ell_{\text{lo}}}^{\comp}$ as follows.  
 \begin{enumerate}[(\textrm{COH}-1)]
 \item For each patch $\Omega^{\ell}_k$ of each level
   $\ell=\ell_{\text{lo}},\ell_{\text{lo}}+1, \ldots, \ell_{\max}$,
 wrap up $\Omega^{\ell}_k$ with two layers of ghost cells
 and fill these ghost cells by AMRCFI
 or by copying data from boxes of the same level,
 or by the ghost-filling formulas (\ref{eq:velBCHomoDiri5th})
 and (\ref{eq:Neumann5}); 
 see Figure \ref{fig:AMRHier}(c).
\item Evaluate $\mathbf{Q}\avg{\phi}^{\ell}$ for each
  $\ell=\ell_{\text{lo}},\ell_{\text{lo}}+1, \ldots, \ell_{\max}$.
 \item If $\mathbf{Q}$ has a divergence form,
   perform refluxing \cite[Fig. 2.1]{zhang2012fourthAMR-ADE}
   for each level $\Omega^{\ell}$ where
   $\ell = \ell_{\max}-1, \ell_{\max}-2, \ldots, \ell_{\text{lo}}$.
 \item Use (\ref{eq:AMR_avergingDownToCrse})
   to average data on $\Omega^{\ell+1}$
   to $\mathcal{C}_r(\Omega^{\ell+1})$
   for each level $\Omega^{\ell}$ where
   $\ell=\ell_{\max}-1, \ell_{\max}-2, \ldots, \ell_{\text{lo}}$.
 \end{enumerate}

The steps (COH-1,2,3,4) generalize
 the discrete operators $\Grad$, $\Div$, and $\Lapl$
 in \eqref{eq:discreteOps}
 to composite operators
 $\Grad_{\ell_{\text{lo}}}^{\comp}$, 
 $\Div_{\ell_{\text{lo}}}^{\comp}$, 
 and $\Lapl_{\ell_{\text{lo}}}^{\comp}$, respectively. 
Consequently, we construct
 a \emph{composite projection} as 
 \begin{equation}
   \label{eq:comp_ProjOp}
   \Proj_{\ell_{\text{lo}}}^{\comp}
   := \Iden
   - \Grad_{\ell_{\text{lo}}}^{\comp}
   \left(\Lapl_{\ell_{\text{lo}}}^{\comp}\right)^{-1}
   \Div_{\ell_{\text{lo}}}^{\comp},
 \end{equation}
 where the inverse of $\Lapl_{\ell_{\text{lo}}}^{\comp}$ is implemented
 by an adaptive multigrid method \cite{zhang2012fourthAMR-ADE}.

For discrete composite projections 
 in second-order AMR 
 \cite{howell1997adaptive,Martin2008cellcenteredadaptiveprojection},
 \mbox{(COH-4)} is often employed as its last step; 
 then the projected composite velocity is always redundant
 even if the composite velocity is not.
 Our composite projection, however, 
 is different.

\begin{lemma}
  \label{lem:compositeProjection}
  For $\Proj_{\ell_{\text{lo}}}^{\comp}$ in (\ref{eq:comp_ProjOp}), 
  the projected velocity $\Proj_{\ell_{\text{lo}}}^{\comp}
  \avg{\mathbf{w}}_{\ell_{\text{lo}}}^{\comp}$ 
  satisfies (\ref{eq:AMR_avergingDownToCrse})
  if and only if the composite velocity
  $\avg{\mathbf{w}}_{\ell_{\text{lo}}}^{\comp}$
  satisfies (\ref{eq:AMR_avergingDownToCrse}).
\end{lemma}
\begin{proof}
  By \mbox{(COH-4)},
  $(\Iden-\Proj_{\ell_{\text{lo}}}^{\comp})
  \avg{\mathbf{w}}_{\ell_{\text{lo}}}^{\comp}$
   satisfies (\ref{eq:AMR_avergingDownToCrse}) anyway.
   Hence
   if a composite velocity
   $\avg{\mathbf{w}}_{\ell_{\text{lo}}}^{\comp}$
   satisfies (\ref{eq:AMR_avergingDownToCrse}),
   so does
   $\Proj_{\ell_{\text{lo}}}^{\comp}
   \avg{\mathbf{w}}_{\ell_{\text{lo}}}^{\comp}$; 
   if $\avg{\mathbf{w}}_{\ell_{\text{lo}}}^{\comp}$
   does not satisfy (\ref{eq:AMR_avergingDownToCrse}),
   nor does
   $\Proj_{\ell_{\text{lo}}}^{\comp}
   \avg{\mathbf{w}}_{\ell_{\text{lo}}}^{\comp}$. 
\end{proof}

We emphasize that,
 inspite of being a concatenation
 by three composite operators in (\ref{eq:comp_ProjOp}), 
 $\Proj_{\ell_{\text{lo}}}^{\comp}$
 is \emph{not} defined by (COH-1,2,3,4).
The distinguishing feature of $\Proj_{\ell_{\text{lo}}}^{\comp}$
 in \Cref{lem:compositeProjection} 
 is crucial 
 for our AMR method
 to achieve fourth-order accuracy;
 see also the last paragraph of \Cref{subsec:single-level-advance}.

\subsection{Regridding and data migration
  (RDM)}
\label{subsec:regridding}

The AMR hierarchy $\Omega_{\Upsilon}$
 may be regridded at any time step
 so that computational resources are focused on
 areas of primary interests. 

\begin{enumerate}[(RDM-1)]
\item For each $\ell=0,1,\ldots,\ell^{\max}\!-\!1$,
  tag cells on $\Omega^{\ell}$
  that satisfy user-supplied criteria such as
  $\|\nabla\times \mathbf{u}\| \geq \epsilon_{\omega}$
  where $\epsilon_{\omega}$ is a user-specified constant.
\item For level $\ell^{\max}-1$,
  group the tagged cells
  into disjoint boxes by the clustering algorithm in 
  \cite{berger1986Data_for_adaptive_grid,berger1991algorithm} 
  and refine these boxes
  to form a new grid $\Omega_{\text{new}}^{\ell^{\max}}$.
\item For each $\ell=\ell^{\max}\!-\!2, \ell^{\max}\!-\!3,\ldots,0$,
  coarsen the boxes in $\Omega_{\text{new}}^{\ell+2}$ twice,
  expand each coarsened box by one cell
  in each direction along each dimension, 
  and obtain a new set of tagged cells
  as the union of these expanded boxes and
  those tagged cells on $\Omega^{\ell}$ in (RDM-1).
  Then,
  apply operations in (RDM-2) to this new set of tagged cells
  to form the new grid $\Omega_{\text{new}}^{\ell+1}$.
\item The new levels obtained in (RDM-2,3) form the new AMR hierarchy.
\end{enumerate}

In (RDM-3), the set union guarantees (\ref{eq:embeddingOfSubdomain})
 while the coarsening, expansion, and refinement of boxes
 imply the proper nesting and refinement conditions.

After generating the new AMR hierarchy,
 we calculate the intersection of the two hierarchies
 and copy data on the common grids
 from the old hierarchy to the new one.
As for cells on the new hierarchy that are outside the common grids,
 we assign their data by AMRCFI
 with those on the old hierarchy as the interpolation source. 

\subsection{Synchronized AMR for INSE}
 \label{sec:an-amr-method-nonsubcycling}

Replace the discrete operators in \eqref{eq:GePUP-IMEX}
 with their composite counterparts, 
 employ the components discussed in the previous subsections, 
 and we obtain a synchronized adaptive method for INSE
 with Dirichlet conditions.
This method can be considered as an extension
 of our previous synchronized AMR algorithms 
 for solving the advection-diffusion equation
 \cite{zhang2012fourthAMR-ADE}
 and INSE with periodic boundary conditions
 \cite{Zhang2014fourthorderapproximateprojection}.
Since the same time step size is used for the entire hierarchy, 
 different levels of composite data
 are \emph{always synchronized at the same time}.

This synchronized AMR algorithm 
 is not the most efficient, 
 since the uniform time step size is restricted
 by numerical stability on the finest level.
Nonetheless, it is useful
 for jump-starting the subcycled AMR; 
 see \Cref{subsec:single-level-advance}.



\section{Algorithms}
\label{sec:algorithm}

In \Cref{sec:gepup-e-level},
 the GePUP-E formulation (\ref{eq:GePUP-E}) is adapted 
 to a single refinement level
 to facilitate subcycling in time. 
In \Cref{subsec:single-level-advance},
 we outline the subcycled AMR method for INSE, 
 with some of its major components
 detailed in the last three subsections.

\subsection{The GePUP-E formulation
  on a single refinement level}
\label{sec:gepup-e-level}

Hereafter we denote
 by $\wbold^{\ell} := \wbold|_{\overline{\dom^{\ell}}}$
 the restriction of a function $\wbold$
 to the $\ell$th subdomain $\dom^{\ell}$. 
\begin{defn}
  \label{def:GePUP-E-OneLevel}
  The \emph{GePUP-E formulation of INSE on the $\ell$th refinement
    level}
  is obtained from \eqref{eq:GePUP-E} by
  replacing $\dom$ and $\partial\dom$
  respectively with $\dom^{\ell}$ and $\partial\dom^{\ell}_{P}$,
  changing $\ubold,\wbold, q$
  respectively to $\ubold^{\ell},\wbold^\ell, q^\ell$, 
  and adding the boundary conditions,
\begin{subequations}
  \label{eq:GePUP-localized}
  \begin{alignat}{3}
    \wbold^{\ell} &
    = \wbold^{I}  
    & \ \mathrm{on}\  &\partial \dom^{\ell}_I, \\
    \ubold^\ell \cdot\nbold &= \ubold^I\cdot\nbold
    & \ \mathrm{on}\  &\partial \dom^{\ell}_I, \\
    \nbold\cdot \nabla q^{\ell} &= \nbold\cdot(\mathbf{g}^{\ell} -\ubold^{I}\cdot\nabla \ubold^{I}
    +\nu\lapl \ubold^{I})
    - \frac{\partial }{\partial t}(\mathbf{n}\cdot\wbold^{I})
    & \quad \ \mathrm{on}\  &\partial \dom^{\ell}_{I},
  \end{alignat}
\end{subequations}
where $\partial \dom^{\ell}_{I}$ and $\partial \dom^{\ell}_{P}$
 are defined in (\ref{eq:subdomainBdryType}),
 $\wbold^{I}$ is obtained
 by restricting the solution $\mathbf{w}$ of (\ref{eq:GePUP-E})
 to $\partial \dom^{\ell}_{I}$,
 and
 $\ubold^{I}:=\left.(\cProjLH\wbold)\right|_{\partial \dom^{\ell}_{I}}$. 
\end{defn}

(\ref{eq:subdomainBoundary}) and (\ref{eq:subdomainBdryType}) 
 imply $\partial \dom^{\ell}_{I}=\emptyset$ for $\ell=0$
 and $\partial \dom^{\ell}_{I}\ne \emptyset$ for $\ell>0$.
Thus \Cref{def:GePUP-E-OneLevel}
 reduces to \Cref{def:GePUP-E} in the case of $\ell=0$.
For $\ell>0$,
 the GePUP-E formulation in Definition \ref{def:GePUP-E-OneLevel}
 is not a standalone problem; 
 it depends on Definition \ref{def:GePUP-E} 
 since the interface conditions in (\ref{eq:GePUP-localized})
 come from the solution of (\ref{eq:GePUP-E}). 

Analogous to (\ref{eq:INSE-solvabilityC}),
 the formulation in \Cref{def:GePUP-E-OneLevel} satisfies,
 on each \emph{connected} component $\dom^{\ell,c}$
 of the $\ell$th subdomain, 
 the compatibility condition
 \begin{equation}
   \label{eq:GePUP-localized-solvability-condition}
   \int_{\partial\dom^{\ell,c}_I} \wbold^I \cdot \nbold 
   +
   \int_{\partial\dom^{\ell,c}_P} \mathbf{u}^b \cdot \nbold 
   = 0.
 \end{equation}

\begin{theorem}
  \label{thm:localHeatEqOfDivW}
  The evolution of $\nabla\cdot\mathbf{w}^{\ell}$
  of the GePUP-E formulation for a single refinement level
  in \Cref{def:GePUP-E-OneLevel}
  is governed by
  \begin{subequations}
    \label{eq:localHeatEqOfDivW}
    \begin{align}
      \frac{\partial\left(\nabla\cdot \mathbf{w}^{\ell}\right)}{\partial t}
      &= \nu\Delta\left(\nabla\cdot \mathbf{w}^{\ell}\right)
      & \ \mathrm{in}\  \dom^\ell, \\
      \mathbf{n}\cdot\nabla\nabla\cdot \mathbf{w}^{\ell} &= 0
      & \ \mathrm{on}\  \partial \dom^{\ell}_I, \\
      \nabla\cdot \mathbf{w}^{\ell} &= 0
      & \ \mathrm{on}\  \partial \dom^{\ell}_P.
    \end{align}
  \end{subequations}
\end{theorem}
\begin{proof}
  (\ref{eq:GePUP-E}e) and the divergence of (\ref{eq:GePUP-E}a)
  imply (\ref{eq:localHeatEqOfDivW}a). 
  $\wbold = \ubold - \grad \phi$ in Definition \ref{def:GePUP-E},
  the commutativity of $\grad$ and $\lapl$, 
  and (\ref{eq:INS}b) yield the identity
  $\lapl\wbold - \grad\div\wbold = \lapl\ubold$, 
  which, together with (\ref{eq:GePUP-localized}c)
  and the normal component of (\ref{eq:GePUP-E}a),
  leads to (\ref{eq:localHeatEqOfDivW}b).
  The Dirichlet condition (\ref{eq:localHeatEqOfDivW}c)
  follows directly from (\ref{eq:GePUP-E}b)
  and is vacuous if $\partial \dom^{\ell}_P=\emptyset$.
\end{proof}

By the maximum principle of the heat equation,
 (\ref{eq:localHeatEqOfDivW}) dictates
 the exponential decay of $\div \wbold^\ell$ on $\dom^{\ell}$,
 a feature similar to that in Theorem \ref{thm:decayOfVelDiv-GePUP-E}. 

 \begin{algorithm}
   \caption{A subcycled AMR method for solving INSE}
\label{alg:time-stepping_procedure}
\hspace*{\algorithmicindent}
\textbf{Input:} Initial time $t_0$, end time $t_e$, time step size $k^0$,
initial AMR hierarchy $\Omega_{\Upsilon}$, 
\\\hspace*{53.5pt}composite velocity $\avg{\wbold}^{\comp}$ that
approximates $\avg{\wbold(t_0)}^{\comp}$\\
\hspace*{\algorithmicindent}
\textbf{Side effect:} Update $\Omega_{\Upsilon}$
and $\langle \mathbf{w}\rangle^{\mathrm{comp}}$ to time $t_e$
\begin{algorithmic}[1]
  \State{Advance $\avg{\wbold}^{\comp}$ to time $t_0 + 3k^0$
    by synchronized AMR with time step size $\dt^{\ell_{\max}}$}
    \Comment{See \Cref{sec:an-amr-method-nonsubcycling}}
  \State{$\forall \ell = 0,1,\ldots,\ell_{\max}$,
    set $k^\ell := k^0/r^{\ell}$,\  
    $t^\ell \gets t_0 + 3k^0 $, and 
    $\mathbf{W}^\ell\leftarrow \mathbf{W}^\ell(t^{\ell})$
    in (\ref{eq:vel_slicing})
  }
  \While{$t^0<t_e$}
  \If{the regridding criteria are satisfied}
  \Comment{See \Cref{subsec:regridding}}
  \State{Regrid 
    and migrate $\avg{\wbold}^{\comp}$ and $\mathbf{W}^{\ell}$
    to the new hierarchy $\Omega_{\Upsilon}$}
  \EndIf
  \State{\Call{\texttt{SingleLevelAdvance}}
    {$0, \{t^m: m\ge 0\}, \langle \mathbf{w}\rangle^{\mathrm{comp}}$}}
  \EndWhile
\end{algorithmic}

\small
\quad\textbf{procedure} \texttt{\textsc{SingleLevelAdvance}} 
($\ell$, $\{t^m: m\ge{\ell}\}$, 
$\langle \mathbf{w}\rangle^{\mathrm{comp}}$)\\
  \hspace*{\algorithmicindent}
  \textbf{Side effect:}
  The composite data $\langle \mathbf{w}\rangle^{\comp}_{\ell}$
  in (\ref{eq:compositeVariable}) 
  is advanced to $t^{\ell}_{\mathrm{sync}}:=t^\ell + \dt^\ell$;
  \\\hspace*{70pt} 
  for each level $m\ge \ell$,
  the current time $t^{m}$ is updated to $t^{\ell}_{\mathrm{sync}}$
\begin{algorithmic}[1]
  \If{$\ell > 0$} \Comment{$\partial\dom^{\ell}_I=\emptyset$
    for the coarsest level}
  \State{Compute, for $\dom^{\ell}$, the interface conditions
    in (\ref{eq:GePUP-localized})}
  \Comment{See \Cref{subsec:compute-coarse-fine}}
  \State{Enforce solvability conditions
    such as (\ref{eq:GePUP-localized-solvability-condition})}
  \Comment{See \Cref{sec:enforc-comp-cond}}
  \EndIf
  \State{Advance $\langle{\wbold}\rangle^{\ell}$
    to $t^{\ell}_{\mathrm{sync}}$
    by IMEX in \eqref{eq:GePUP-IMEX} with time step size
      $\dt^{\ell}$}
  \Comment{See \Cref{subsec:line-solv-irreg}}
  \State{Set $t^{\ell} \gets t^{\ell}_{\mathrm{sync}}$
    and $\mathbf{W}^\ell\gets
    \mathbf{W}^\ell(t^{\ell}_{\mathrm{sync}})$}
  \Comment{See (\ref{eq:vel_slicing})}
  \If{$\ell<\ell_{\max}$}
  \While{$t^{\ell+1}< t^{\ell}_{\mathrm{sync}}$}
  \State{\Call{\texttt{SingleLevelAdvance}}
    {$\ell+1$, $\{t^m: m\ge \ell+1\}$, 
      $\langle \mathbf{w}\rangle^{\mathrm{comp}}$}}
  \Comment{The recursion}
  \EndWhile
  \State{Set $\langle{\wbold}\rangle^{\comp}_{\ell}
    \leftarrow \Proj_{\ell}^{\comp} \langle{\wbold}\rangle^{\comp}_{\ell}$
    with $\Proj_{\ell}^{\comp}$ in (\ref{eq:comp_ProjOp})}
  \Comment{The synchronization}
  \EndIf
\end{algorithmic}



 \end{algorithm}
 \begin{figure}
   \label{fig:coarseFineInterface-Interp}
   \centering
   \subfigure[Vary time step sizes across levels
   in $\Omega_{\Upsilon}$
   with $k^{\ell}=r k^{\ell+1}$]
   {
     \includegraphics[scale=0.36]{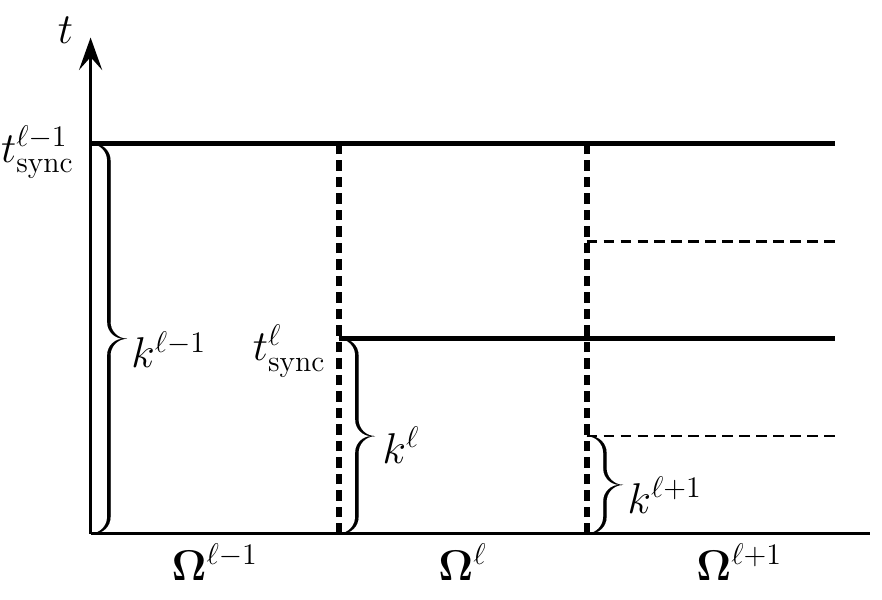}
   }
   \hfill
   \subfigure[Compute interface conditions on $\partial \dom^{\ell}_I$
   in (\ref{eq:GePUP-localized})
   from cell averages in $\Omega^{\ell-1}$ 
   ]
   {
     \includegraphics[scale=0.62]{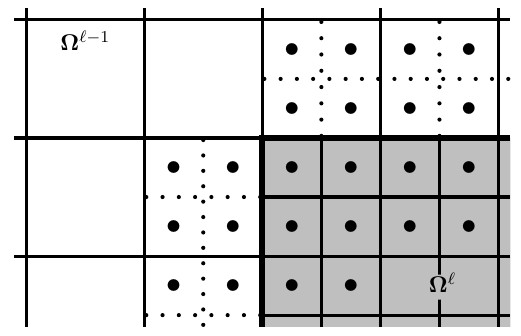}
   }
   \\
   \subfigure[Interpolate 
   interface conditions in (\ref{eq:GePUP-localized})
   on $\partial \dom^{\ell}_I$ at $t^*$
   from data at $t_{-i}:=t^{\ell-1}-ik^{\ell-1}$]
   {
     \includegraphics[scale=0.95]{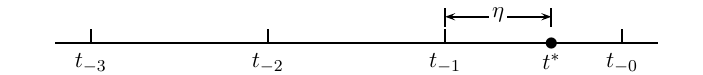}
   }
   \caption{Key components of subcycled AMR.
     In (a),
     a vertical dashed line represents a coarse-fine interface,
     and a horizontal thick solid line indicates 
     a time of synchronization. 
     In (b),
     the shaded area represents $\dom^{\ell}$, 
     thick solid lines 
     $\partial \dom^{\ell}_I$ in (\ref{eq:subdomainBdryType}),
     and ``$\bullet$'' a cell average near $\partial \dom^{\ell}_I$
     for calculating face averages
     on $\partial \dom^{\ell}_I$. 
     In (c), interpolation source data
     are prepared at $t_{-i}$ where $i=0,1,2,3$.
     }
 \end{figure}

\subsection{Subcycled AMR for INSE} 
\label{subsec:single-level-advance}

Different from the synchronized AMR
 in \Cref{sec:an-amr-method-nonsubcycling}, 
 the subcycled AMR consists of
 a sequence of recursive single-level advances,
 with the time step size $k^{\ell}$ dependent on the level index $\ell$.
As illustrated in \Cref{fig:coarseFineInterface-Interp}(a), 
 two adjacent levels satisfy $k^{\ell}= r k^{\ell+1}$
 where $r$ is the spatial refinement ratio
 introduced in the first paragraph
 of \Cref{sec:amr-hierarchy}.
Since the time step size of a coarse level
 is proportionally larger than that of its finer level,
 the efficiency of subcycled AMR
 is better than that of synchronized AMR.
 
Our subcycled AMR method is outlined 
 in \Cref{alg:time-stepping_procedure}. 
At line 1, the synchronized AMR in \Cref{sec:an-amr-method-nonsubcycling}
 is employed to advance the composite velocity $\avg{\wbold}^{\comp}$
 from $t_0$ to $t_0+3k^0$ 
 with the time step size $\dt^{\ell_{\max}}=k^0/r^{\ell_{\max}}$.
These results are used at line 2 to initialize
 $\mathbf{W}^{\ell}(t_0+3k^0)$
 where the set $\mathbf{W}^{\ell}(\tau)$ is defined as 
 \begin{equation}
   \label{eq:vel_slicing}
   \forall \ell = 0,1, \ldots, \ell_{\max}-1,\ \ 
   \mathbf{W}^{\ell}(\tau) := \left\{ \avg{\wbold(t)}^{\ell}:
     t = \tau-i\dt^{\ell} \text{ where } i=0,1,2,3
   \right\},  
 \end{equation}
 which serves as the source data for temporal interpolation
 in computing the coarse-fine interface conditions
 in (\ref{eq:GePUP-localized}). 
Note that, at line 3, $t^0$ is the variable representing
 the current time of $\avg{\mathbf{w}}$ on the coarsest level
 and is different from $t_0$,
 the constant input parameter that denotes the initial time.

Given the velocity $\avg{\wbold}^{\comp}$
 and a base level index $\ell$,
 the procedure \textsc{SingleLevelAdvance}
 advances each level in $\avg{\wbold}^{\comp}_{\ell}$
 from $t^{\ell}$ to $t^{\ell}_{\mathrm{sync}}:=t^{\ell} + \dt^{\ell}$
 as follows. 
First, the velocity $\avg{\wbold}^{\ell}$ of the $\ell$th level
 is advanced to $t^{\ell}_{\mathrm{sync}}$ 
 in one time step, 
 with lines 2, 3, 5 detailed in Subsections
 \ref{subsec:compute-coarse-fine},
 \ref{sec:enforc-comp-cond},
 and \ref{subsec:line-solv-irreg}, respectively.
Then each level finer than $\ell$ is recursively advanced
 by lines 8--10 to $t^{\ell}_{\mathrm{sync}}$
 in multiple time steps. 
Finally at line 11,
 the level velocities in $\avg{\wbold}^{\comp}_{\ell}$
 are synchronized 
 by the composite projection in (\ref{eq:comp_ProjOp}).
 
Due to \Cref{lem:compositeProjection}
 and the decoupling of level advances, 
 neither at line 12 
 nor at line 10
 does the velocity 
 on the invalid region of a coarse level
 satisfy the redundancy condition
 in (\ref{eq:AMR_avergingDownToCrse}). 
Hence
 the synchronization should be interpreted
 not in the sense of redundancy 
 but with respect to the consistency
 of fulfilling the divergence-free condition across multiple levels.
It is emphasized that
 we never apply (COH-4) in \Cref{sec:AMR_composite-operators}
 to $\avg{\mathbf{w}}^{\comp}_{\ell}$
 after line 12.  
This makes our AMR method prominently different
 from second-order AMR methods
 \cite{howell1997adaptive,Martin2008cellcenteredadaptiveprojection},
 in which (COH-4) is applied as an extra step
 to $\Proj_{\ell}^{\comp}\avg{\mathbf{w}}^{\comp}_{\ell}$.
As demonstrated in \Cref{subsec:test-viscous-box},
 applying this extra step 
 to $\Proj_{\ell}^{\comp}\avg{\mathbf{w}}^{\comp}_{\ell}$
 reduces our AMR method to third-order accuracy
 with dominating solution errors concentrated
 at the coarse-fine interface. 
In comparison, 
 adhering to the formula (\ref{eq:comp_ProjOp})
 yields a better continuity
 of the coarse velocity at the coarse-fine interface
 and recovers fourth-order convergence rates of our method.
Therefore,
 (\ref{eq:comp_ProjOp}) without (COH-4)
 is the suitable form of discrete composite projection
 for subcycled AMR. 

\subsection{Computing coarse-fine interface conditions
 (CIC) in (\ref{eq:GePUP-localized})}
 \label{subsec:compute-coarse-fine}

Referring to \Cref{fig:coarseFineInterface-Interp}(c), 
 we write Talor expansions of a function
 $\varphi$ and $\frac{\dif\varphi}{\dif t}$ at $t_{-i}$'s as
 \begin{equation}
   \label{eq:timeInterp}
   \left[
     \begin{matrix}
       \varphi(t^*), 
       k\frac{\dif\varphi}{\dif t}(t^*)
     \end{matrix}
   \right]^{\top}
   =
   \mathbf{I}^{(4)}(k,\eta)
   \left[
     \begin{matrix}
       \varphi(t_{-0}), 
       \varphi(t_{-1}), 
       \varphi(t_{-2}), 
       \varphi(t_{-3})
     \end{matrix}
   \right]^{\top} + O(k^4),
 \end{equation}
 where $k:=k^{\ell-1}$, $\eta:=t^*-t_{-1}$, 
 and the fourth-order interpolation matrix
 $\mathbf{I}^{(4)}(k,\eta)$ is
 \begin{displaymath}
   \small
   \begin{aligned}
  \renewcommand{\arraystretch}{1.2}
  \mathbf{I}^{(4)}(k,\eta) \!=\!
  \frac{1}{6k^3}\!
  \left[
    \begin{matrix}
      \!-\eta^3 \!+\! k^2\eta & 3\eta^3\!+\!3k \eta^2\!-\!6k^2\eta & \!-3\eta^3\!-\!6k \eta^2\!+\!3k^2\eta\!+\!6k^3 & \eta^3 \!+\! 3k \eta^2 \!+\!2k^2\eta \\
      \!-3\eta^2\!+\!k^2 & 9\eta^2\!+\!6k \eta\!-\!6k^2    & \!-9\eta^2\!-\!12k \eta \!+\! 3k^2     & 3\eta^2\!+\!6k \eta\!+\!2k^2      \\
    \end{matrix}
  \right]
\end{aligned}
.
 \end{displaymath}

To prepare for the advance of $\avg{\wbold}^{\ell}$
 in \Cref{def:GePUP-E-OneLevel}
 by the IMEX scheme in \eqref{eq:GePUP-IMEX},
 we compute face averages of (\ref{eq:GePUP-localized})
 at the coarse-fine interface $\partial \dom^{\ell}_I$ as follows.
\begin{enumerate}[({CIC}-1)]
\item Locate, for $\partial \dom^{\ell}_I$,
  two layers of nearby fine cells in $\Omega^{\ell}$;
  see Figure \ref{fig:coarseFineInterface-Interp}(b).
\item Let $\phi$ denote
  $\wbold$, $\ubold$, $\ubold\cdot\grad\ubold$, and $\lapl \ubold$. 
  For each $i=0,1,2,3$,
  evaluate coarse cell averages  
  $\avg{\phi(t^{\ell-1}_{-i})}^{\ell-1}_{\ibold}$
  by applying the discrete operators in \Cref{sec:spat-discr}
  to $\avg{\wbold(t^{\ell-1}_{-i})}^{\ell-1}_{\ibold}$ in
  $\mathbf{W}^{\ell-1}$,
  use AMRCFI to calculate fine cell averages
  $\avg{\phi(t^{\ell-1}_{-i})}^{\ell}_{\ibold}$
  from
  $\avg{\phi(t^{\ell-1}_{-i})}^{\ell-1}_{\ibold}$,
  and use \eqref{eq:cell2face} to 
  convert fine cell averages to face averages
  on the interface.
\item For each stage $s=2,3, \ldots,\nStages$
  in (\ref{eq:GePUP-IMEX}),
  use (\ref{eq:timeInterp}) to obtain face averages
  $\avg{\phi}^{\ell}_{\ipmfed{1}}$ and
  $\avg{\frac{\partial \wbold}{\partial t}}^{\ell}_{\ipmfed{1}}$
  at  $t^*:= t^\ell + c_s\dt^{\ell}$,
  and calculate the right-hand side (RHS) of (\ref{eq:GePUP-localized}),
  where $c_s$ is the same as that below (\ref{eq:GePUP-IMEX}).
\end{enumerate} 

For a time integrator with dense output
\cite{mccorquodale2011high,christopher2022high}, 
 $\mathbf{W}^{\ell}$ can be removed as unnecessary.
However,
 the inclusion of $\mathbf{W}^{\ell}$
 frees the subcycled AMR 
 from being restricted
 to time integrators with dense output,
 leading to a better generality of our method.
 
\subsection{Enforcing solvability conditions}
\label{sec:enforc-comp-cond}

There is no truncation error in
 the following discrete version of the compatibility condition
 in (\ref{eq:GePUP-localized-solvability-condition}),
 \begin{equation}
   \label{eq:velocity_disc_solvabilityCond}
   \sum_{\ipmfed{1} \in \partial \dom^{\ell,c}_I}
   \avg{\nbold\cdot\wbold^I}^{\ell}_{\ipmfed{1}}
   +
   \sum_{\ipmfed{1} \in \partial \dom^{\ell,c}_P}
   \avg{\nbold\cdot\mathbf{u}^b}^{\ell}_{\ipmfed{1}}
   = 0,  
 \end{equation}
 which is also the solvability condition
 for $\avg{\mathbf{u}}=\mathbf{P}\avg{\mathbf{w}}$
 in (\ref{eq:GePUP-IMEX}) on each connected region
 of the $\ell$th subdomain.
 
When we approximate $\avg{\nbold\cdot\wbold^I}^{\ell}_{\ipmfed{1}}$
 in (\ref{eq:velocity_disc_solvabilityCond})
 by the face average
 $\avg{\mathbf{n}\cdot\mathbf{w}^I}^{\ell,*}_{\ipmfed{1}}$
 calculated from (CIC-1,2,3), 
 we incur a finite error, albeit small, 
 to the solvability condition
 of $\avg{\mathbf{u}}=\mathbf{P}\avg{\mathbf{w}}$.
To annihilate this error,
we correct the approximate face averages by
 $\avg{\mathbf{n}\cdot\mathbf{w}^I}^{\ell,c}_{\ipmfed{1}}
 = \avg{\mathbf{n}\cdot\mathbf{w}^I}^{\ell,*}_{\ipmfed{1}}
 +\delta^c_{\mathbf{w}}$
 where $\delta^c_{\mathbf{w}}$ satisfies
\begin{equation}
  \label{eq:correctingFaceAvg}
  -\delta^c_{\mathbf{w}} |\partial\dom^{\ell,c}_I| =
  \sum_{\ipmfed{1} \in \partial \dom^{\ell,c}_I}
  \avg{\nbold\cdot\wbold^I}^{\ell,*}_{\ipmfed{1}}
  +
  \sum_{\ipmfed{1} \in \partial \dom^{\ell,c}_P}
  \avg{\nbold\cdot\mathbf{u}^b}^{\ell}_{\ipmfed{1}}
\end{equation}
 and $|\partial\dom^{\ell,c}_I|$ is the number of faces
 that partition $\partial\dom^{\ell,c}_I$.

Similar corrections are applied to
 $\avg{\mathbf{n}\cdot\mathbf{u}^I}^{\ell,*}_{\ipmfed{1}}$
 in (\ref{eq:GePUP-localized}b)
 and the RHS of (\ref{eq:GePUP-localized}c)
 to enforce the divergence-free condition (\ref{eq:INS}b)
 and the solvability condition
 of the discrete PPE in (\ref{eq:GePUP-IMEX}),
 respectively. 
 
\subsection{Geometric multigrid for
 three types of linear systems (TLS)}
\label{subsec:line-solv-irreg}

At line 5 of \textsc{SingleLevelAdvance}
 in \Cref{alg:time-stepping_procedure}, 
 we have to solve three types of linear systems
 in (\ref{eq:GePUP-IMEX}): 
 \begin{enumerate}[(TLS-1)]
 \item discrete PPEs for $\avg{q}$
   with pure Neumann boundary conditions,
 \item Helmholtz-like linear systems for $\avg{\mathbf{w}}$, 
 \item equations of the form
   $\mathbf{L}\avg{\mathbf{u}}=\mathbf{D}\avg{\mathbf{w}}$
   to implement the discrete projection. 
 \end{enumerate}

All the above linear equations are 
 solved by a geometric multigrid method \cite{briggs2000multigrid}, 
 with injection as the restriction operator,
 linear interpolation as the prolongation operator,
 and the weighted Jacobi ($\omega = \frac{2}{3}$) as the smoother.

The weighted Jacobi also serves
 as an excellent bottom solver for (TLS-2), 
 since the strict diagonal dominance of the matrices implies
 that the spectral radius of the iteration matrix
 is much smaller than one.
However, for \mbox{(TLS-1)} and (TLS-3), 
 the weighted Jacobi is no longer a good bottom solver
 because the spectral radius of the iteration matrix
 is close to one.
In addition,
 the dynamic nature of AMR may
 create a refinement level whose constituting boxes
 have complex adjacency
 and unusual shapes such as a long strip.
In these cases, 
 we switch the bottom solver 
 to the ILU-preconditioned GMRES($m$) method
 \cite[Chapter 9]{saad2003iterative}
 and obtain an optimal complexity 
 of the corresponding geometric multigrid method.
 
As shown in Figure \ref{fig:mg_convergenceRate},
 the reduction rates of one V-cycle
 are around $10$ and $10^4$
 for (TLS-1,3) and (TLS-2), respectively.
The high efficiency in the case of (TLS-1,3)
 is due to the following reasons.
First, as part of the bottom solver,
 the ILU factorization is \emph{only} applied to the linear system
 associated with the coarsest multigrid level. 
Second, the preconditioning matrix derived in the ILU factorization
 well approximates the coefficient matrix of linear system
 and yields faster GMRES convergence.
In our tests,
 only three to five GMRES iterations are required to
 boost the residual reduction rates to around 10.
Last, after each regridding,
 the ILU factorization is executed only once,
 and the resulting triangular matrices are stored and reused
 in GMRES iterations.
Then each preconditioning only entails
 a matrix-vector multiplication, 
 which is highly efficient
 because the matrices resulting from the ILU factorization 
 inherit the sparsity of the original matrix.



\section{Numerical Tests}
\label{sec:numerical-tests}

In this section, we perform a number of benchmark tests
 to demonstrate the accuracy and efficiency
 of the proposed AMR method.
We calculate cell averages of $\mathbf{u}_0$
 with a sixth-order Newton-Cotes quadrature formula, 
 set $\avg{\mathbf{w}_0} = \avg{\mathbf{u}_0}$, 
 and advance $\avg{\mathbf{w}}$
 from initial time $t_0$ to the final time $t_e$ with
 a specified Courant number
 Cr := $\frac{k}{h}\|\mathbf{u}\|_{\infty}$. 

Based on valid regions in (\ref{eq:validRegion}), 
 the \emph{$L_p$-norm of composite data} is defined as
 \begin{equation}
   \left\|\avg{\varphi}^{\comp} \right\|_{p}
   := \left\{
   \begin{array}{ll}
     \max_{\ell\in[0,\ell_{\max}]}
     \max_{\ibold\in\Omega^\ell_{\mathrm{valid}}}
     \left|\avg{\varphi}^{\ell}_{\ibold}\right|
     & \text{if } p=\infty;
     \\
     \left(
     \sum_{\ell\in[0,\ell_{\max}]} (h^{\ell})^{\Dim}
     \sum_{\ibold\in\Omega^\ell_{\mathrm{valid}}}
     |\avg{\varphi}^{\ell}_{\ibold}|^p 
     \right)^{\frac{1}{p}}
     & \text{if } p\in \mathbb{N}^+, 
   \end{array}
   \right.
 \end{equation}
 where $|\avg{\varphi}^{\ell}_{\ibold}|$ is 
 the absolute value of the cell average of $\varphi$
 over the $\ibold$th control volume of the $\ell$th level.

If the exact solution of a test is available, 
 we calculate composite errors
 by subtracting the computed solution from the exact solution,
 compute the $L_p$ norm of composite errors
 by (\ref{eq:compositeErrorDynamic}), 
 and deduce the corresponding convergence rates
 from the $L_p$ norms on successively refined AMR hierarchies.

When no exact solution is available for a test, 
 we determine composite solution errors as follows.
If $\ell_{\max}=0$, i.e.,
 each AMR hierarchy contains only one level,
 solution errors are obtained
 by standard Richardson extrapolation. 
If $\ell_{\max}\ge 1$
 and the set of subdomains is \emph{static}, 
 i.e.,
 $\{\dom^{\ell}: \ell=1,\ldots,\ell_{\max}\}$ does not depend on time,
 we perform the test on three or more AMR hierarchies
 with the same set of subdomains,
 the same refinement ratio, 
 and successively refined base levels; 
 then the \emph{composite solution error
   between any pair of adjacent AMR hierarchies}
 is the composite data on the coarse AMR hierarchy
 in which each level of errors
 is obtained by Richardson extrapolation
 on the two corresponding levels of solutions. 
 
A more common scenario of AMR is described by $\ell_{\max}\ge 1$
 and the set 
 of subdomains being \emph{dynamic}: 
 each subdomain $\dom^{\ell}$ changes on the fly at the runtime
 and thus may vary from one AMR hierarchy to another.
In this case,
 we perform the test on a very fine single-level grid
 and denote its solution as $\avg{\phi}_{\mathrm{ref}}$;
 then the composite error of the solution $\avg{\phi}^{\comp}$
 on an AMR hierarchy is defined by 
 \begin{equation}
   \label{eq:compositeErrorDynamic}
   \avg{e_{\phi}}^{\comp}
   := \avg{\phi}_{\mathrm{ref}}^{\comp} - \avg{\phi}^{\comp},
 \end{equation}
 where each level data $\avg{\phi}_{\mathrm{ref}}^{\ell}\in
 \avg{\phi}_{\mathrm{ref}}^{\comp}$
 is obtained by coarsening $\avg{\phi}_{\mathrm{ref}}$
 to $\Omega^{\ell}$. 
Then the convergence rates are estimated
 from the $L_p$ norms of $\avg{e_{\phi}}^{\comp}$
 by a modified Richardson extrapolation
 \cite[Section 5.4]{zhang2012fourthAMR-ADE}
 to counteract the effect of
 $\avg{\phi}_{\mathrm{ref}}$ not being the exact solution. 
Since two adjacent hierarchies may have different subdomains
 for each refinement level, 
 the convergence rates
 obtained for dynamic subdomains
 are not intended to verify the order of accuracy
 of an AMR method.
Instead, they measure how well
 the leading solution errors have been reduced
 by dynamic AMR.
 
In all tests,  
 the penalty parameter in (\ref{eq:GePUP-E}f)
 is set to  $\lambda=1$
 and the pressure is extracted from the Eulerian accelerations
 $\mathbf{a}:= \frac{\partial \mathbf{u}}{\partial t}$ and
 $\mathbf{a}^*:= \mathbf{g} + \nu\Delta \ubold -
 \ubold\cdot\nabla\ubold$
 \mbox{\cite[Section 4.3]{zhang2016gepup}}, 
 i.e., we rewrite (\ref{eq:INS}a)
 as $\mathbf{a}^* = \mathbf{a} + \nabla p$
 and solve for $p$ from the PPE
 $\Delta p = \nabla \cdot \mathbf{a}^*$
 with the Neumann boundary condition
 that results from (\ref{eq:INS}c)
 and the normal component of (\ref{eq:INS}a).

\subsection{The Taylor--Green vortex test}

Our first test is the Taylor--Green vortex
 \cite{taylor37:_mechan}
 on the domain $\mathcal{D} = [0,2]^2$
 with analytic solutions
 \begin{subequations}
  \label{eq:decayingSin-exactSol}
   \begin{align}
     \ubold(x,y,t) &= \exp(-2\pi^2\nu t)
     \left(
     \begin{matrix}
       -\cos(\pi x) \sin(\pi y) \\
       \sin(\pi x) \cos(\pi y)
     \end{matrix}
     \right),\\
     p(x, y, t) &= -\frac{1}{4} \exp(-6\pi^2\nu t)
     \bigg(
     \cos(2\pi x) + \cos(2\pi y)
     \bigg)
   \end{align}
 \end{subequations}
 where $\nu = 0.01$. 
The time derivative of the velocity cancels the diffusion term and
 the pressure gradient cancels the convection term,
 resulting in a zero external force.
We set $\ubold^b=\left.\ubold\right|_{\partial\dom}$
 as the Dirichlet boundary condition of the velocity
 and use cell averges of $\ubold(x,y,t_0)$
 in (\ref{eq:decayingSin-exactSol})
 as the initial condition.

We advance $\avg{\mathbf{w}}$ from $t_0=0$ to $t_e=1$
 with Cr $=0.1$
 on a static AMR hierarchy $\Omega_{\Upsilon}$
 with $\mathcal{D}^{1}=[0.5,1.5]^2$
 and $r=2$.
The composite errors at $t_e$ are calculated
 by subtracting the computed solutions from 
 cell averages of $\avg{\mathbf{u}(x,y,t_e)}$
 and $\avg{p(x,y,t_e)}$. 

Error norms and convergence rates of the proposed method
 are listed in Table \ref{tab:velIConverRate_suitableRefined}, 
 where fourth-order convergence rates in all norms 
 are clearly observed,
 verifying the correctness and the high accuracy
 of our subcycled AMR method in \Cref{alg:time-stepping_procedure}. 
In particular,
 neither the subcycling in time
 nor the spatial interpolation at the coarse-fine interface
 degrades the fourth-order accuracy.
 
\begin{table}
  \label{tab:velIConverRate_suitableRefined}
  \caption{Errors and convergence rates
      of the subcycled AMR method
      in \Cref{alg:time-stepping_procedure}
      for solving the Taylor--Green test
      (\ref{eq:decayingSin-exactSol})
      with $\mathrm{Re}=100$, $t_e=1$ and  $\mathrm{Cr}=0.1$.}
    \begin{tabular}{p{2.8cm}|c|c|c|c|c|c}
  \hline \hline
  \multicolumn{2}{c|}{static AMR} 
  & $h^{0}=\frac{1}{64}$ & rate
  & $h^{0}=\frac{1}{128}$& rate
  & $h^{0}=\frac{1}{256}$ 
  \\ \hline
  \multirow{6}{75pt}{$\ell_{\max}=1$; $r=2$; $\dom^1=[0.5,1.5]^2$}
  & $\ubold$ \ $L_{\infty}$ & 2.39e-06 & 3.95 & 1.55e-07 & 3.97 & 9.90e-09
  \\ 
  & $\ubold$ \ $L_1$ & 6.83e-07 & 3.98 & 4.31e-08 & 4.00 & 2.69e-09
  \\ 
  & $\ubold$ \ $L_2$  & 9.71-07 & 3.99 & 6.11e-08 & 4.01 & 3.78e-09
  \\ \cline{2-7}
  & $p$ \ $L_{\infty}$ & 1.47e-05 & 4.02 & 9.06e-07 & 4.00 & 5.60e-08
  \\ 
  & $p$ \ $L_1$   &  4.52e-06 & 4.04 & 2.75e-07 & 4.01 & 1.71e-08
  \\ 
  & $p$ \ $L_2$ & 6.01e-06 & 4.06 & 3.60e-07 & 3.99 & 2.26e-08
  \\ \hline \hline
\end{tabular}

\end{table}
 
\subsection{The viscous box test with Re=100}
 \label{subsec:test-viscous-box}

As in \cite{zhang2016gepup},
 the domain $\dom=[0,1]^2$ of this test has
 the no-slip boundary condition $\mathbf{u}^b=\mathbf{0}$
 and the initial velocity is 
\begin{equation}
  \label{eq:viscousBox}
  \mathbf{u}_0(x,y) 
  = 
  \left(
  \begin{array}{c}
	\, \sin^2(\pi x) \, \sin(2 \pi y) \,
    \\
  	\,  -\sin(2 \pi x) \, \sin^2(\pi y) \,
  \end{array}
  \right). 
\end{equation}

We advance $\avg{\mathbf{w}}$
 from $t_0=0$ to $t_e=0.5$ with Courant number $\mathrm{Cr} = 0.1$.
For AMR, static subdomains are used with $\ell^{\max}=1$.
The composite solution errors are calculated
 by Richardson extrapolation.

\begin{table}
  \label{tab:viscous-box}
  \centering
  \caption{Solution errors and convergence rates
    of the proposed method for the viscous box test
    with $\mathrm{Re}=100$, $t_0=0.0$,
    $t_e=0.5$, and  $\mathrm{Cr}=0.1$.
    For (b) and (c), 
    the subdomain $\dom^1$ of the refined level $\Omega^1$
    is static with $r=2$
    and consists of four squares of size $\frac{1}{8}$
    that are adjacent to the domain corners; 
    see Figure \ref{fig:pseudoColor-viscousbox-errors} (b).
    The composite data of solution errors
    are calculated via Richardson extrapolation. 
  }
  \begin{tabular}{p{3.3cm}|c|c|c|c|c|c}
  \hline \hline
  \multicolumn{2}{c|}{Grid size $h^0$ of the coarsest level}
  & $\frac{1}{64}-\frac{1}{128}$ & rate & $\frac{1}{128}-\frac{1}{256}$ & rate & $\frac{1}{256}-\frac{1}{512}$
  \\ \hline
  \multirow{6}{90pt}{(a) $\ell_{\max}=0$, i.e., single-level grids}
  & \hspace{4.19pt}$\ubold$ \ $L_{\infty}$ & 7.86e-06 & 2.38 & 1.51e-06 & 2.76 & 2.23e-07
  \\ 
  & $\ubold$ \ $L_1$ & 2.03e-06 & 3.96 & 1.30e-07 & 4.00 & 8.12e-09
  \\ 
  & $\ubold$ \ $L_2$ & 2.62e-06 & 3.94 & 1.71e-07 & 3.96 & 1.10e-08
  \\ \cline{2-7}
  & \hspace{4.19pt}$p$ \ $L_{\infty}$ & 2.04e-05 & 2.16 & 4.54e-06 & 1.19 & 1.99e-06
  \\ 
  & $p$ \ $L_1$ & 1.89e-06 & 3.72 & 1.43e-07 & 3.19 & 1.56e-08
  \\ 
  & $p$ \ $L_2$ & 3.04e-06 & 3.59 & 2.53e-07 & 2.71 & 3.85e-08
  \\ \hline
  \multirow{6}{90pt}{(b) $\ell_{\max}=1$ with static $\dom^1$
  and no (COH-4) after $\mathbf{P}^{\comp}_0$
  in (\ref{eq:comp_ProjOp})}
  & \hspace{4.19pt}$\ubold$ \ $L_{\infty}$ & 7.26e-06 & 3.91 & 4.82e-07 & 3.77 & 3.53e-08
  \\ 
  & $\ubold$ \ $L_1$ & 1.84e-06 & 3.99 & 1.15e-07 & 4.03 & 7.03e-09 
  \\ 
  & $\ubold$ \ $L_2$ & 2.43e-06 & 3.99 & 1.53e-07 & 4.03 & 9.38e-09
  \\ \cline{2-7}
  & \hspace{4.19pt}$p$ \ $L_{\infty}$ & 2.06e-05 & 3.25 & 2.16e-06 & 1.80 & 6.18e-07
  \\ 
  & $p$ \ $L_1$ & 1.80e-06 & 3.74 & 1.35e-07 & 3.16 & 1.51e-08 
  \\ 
  & $p$ \ $L_2$ & 2.98e-06 & 3.78 & 2.17e-07 & 3.27 & 2.25e-08
  \\ \hline
  \multirow{6}{90pt}{(c) $\ell_{\max}=1$ with static $\dom^1$
  and $\mathbf{P}^{\comp}_0$ always followed by (COH-4)}
  & \hspace{4.19pt}$\ubold$ \ $L_{\infty}$  & 5.59e-05 & 2.98 & 7.11e-06 & 3.02 & 8.76e-07
  \\ 
  & $\ubold$ \ $L_1$ & 6.59e-06 & 3.01 & 8.19e-07 & 3.00 & 1.02e-07  
  \\ 
  & $\ubold$ \ $L_2$ & 1.11e-05 & 2.98 & 1.41e-06 & 3.01 & 1.74e-07
  \\ \cline{2-7}
  & \hspace{4.19pt}$p$ \ $L_{\infty}$ & 2.15e-05 & 2.56 & 3.65e-06 & 2.16 & 8.17e-07
  \\ 
  & $p$ \ $L_1$ & 4.39e-06 & 3.16 & 4.93e-07 & 3.07 & 5.84e-08
  \\ 
  & $p$ \ $L_2$ & 4.93e-05 & 3.07 & 5.86e-06 & 3.04 & 7.13e-07
  \\ \hline  \hline
\end{tabular}


\end{table}

\begin{figure}
  \centering
  \subfigure[$\ell_{\max}=0$]
  {
    \includegraphics[scale=0.112]{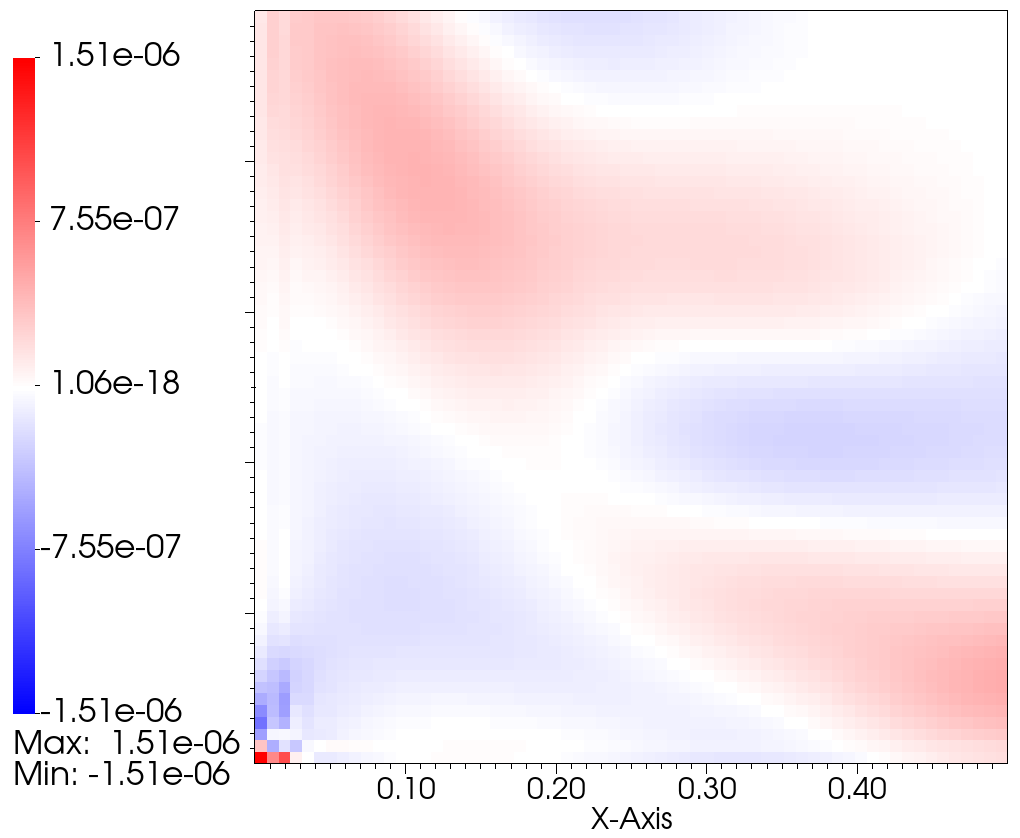}
  }
  \hfill
  \subfigure[$\ell_{\max}=1$ and no \mbox{(COH-4)}]
  {
    \includegraphics[scale=0.112]{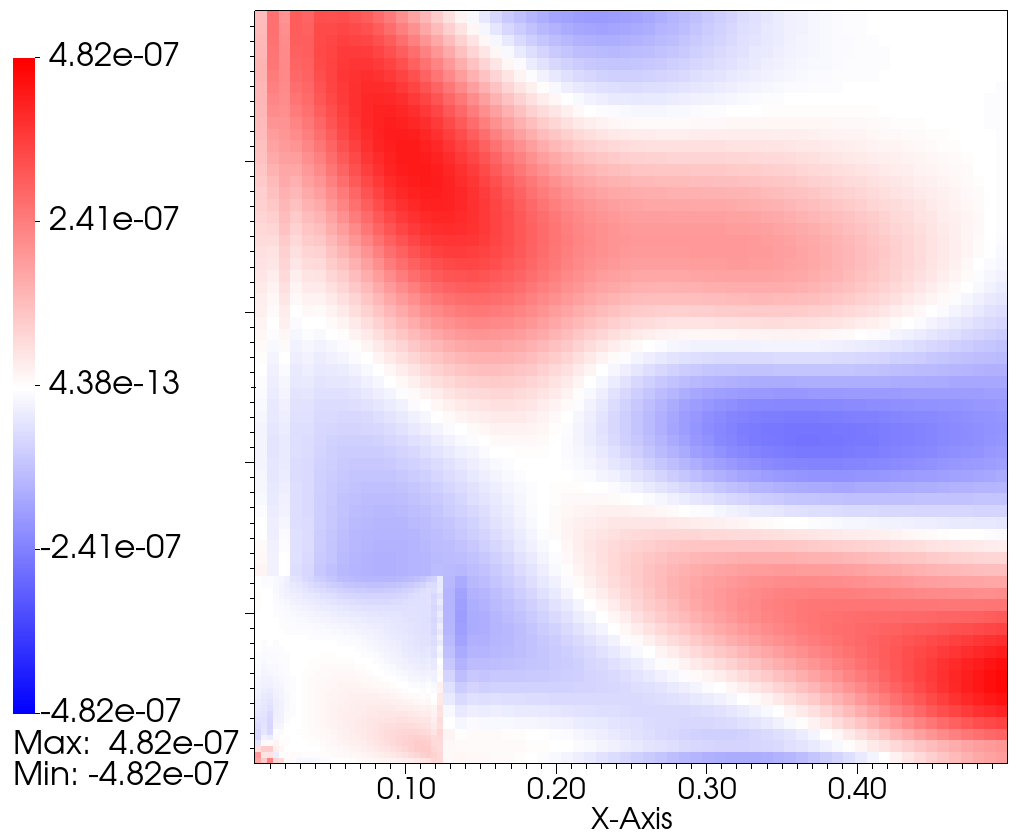}
  }
  \hfill
  \subfigure[$\ell_{\max}=1$ and \mbox{(COH-4)}]
  {
    \includegraphics[scale=0.112]{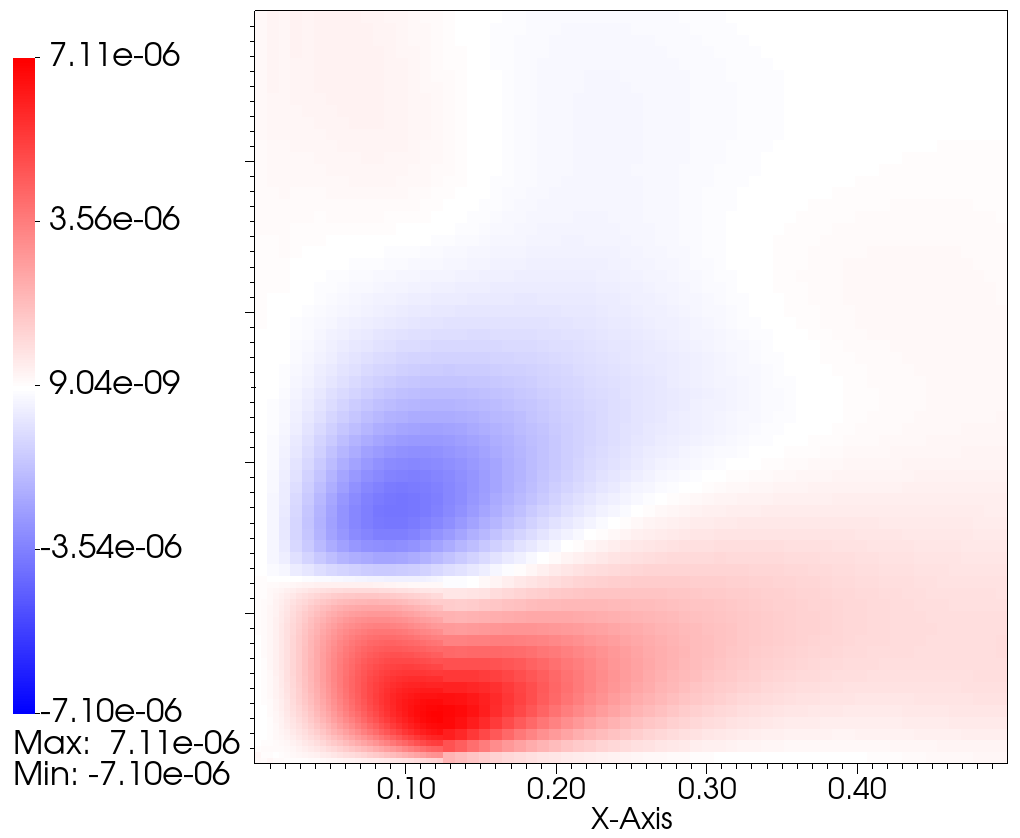}
  }
  
  \subfigure[$\ell_{\max}=0$]
  {
    \includegraphics[scale=0.112]{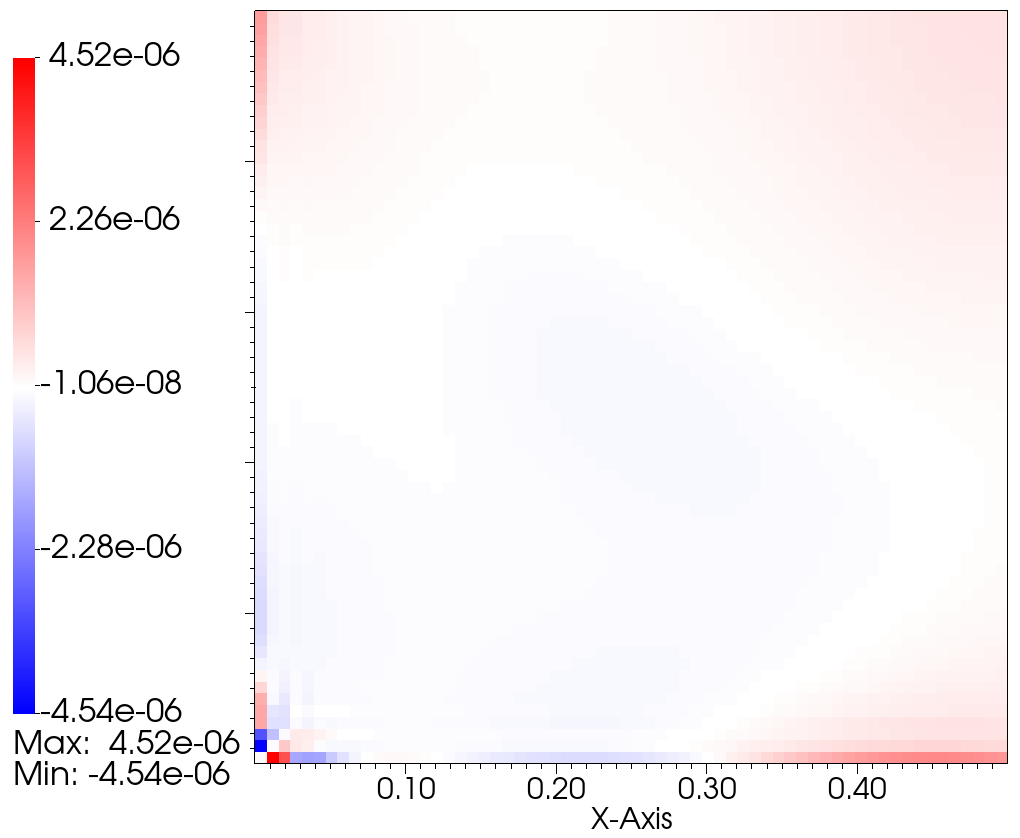}
  }
  \hfill
  \subfigure[$\ell_{\max}=1$ and no \mbox{(COH-4)}]
  {
    \includegraphics[scale=0.112]{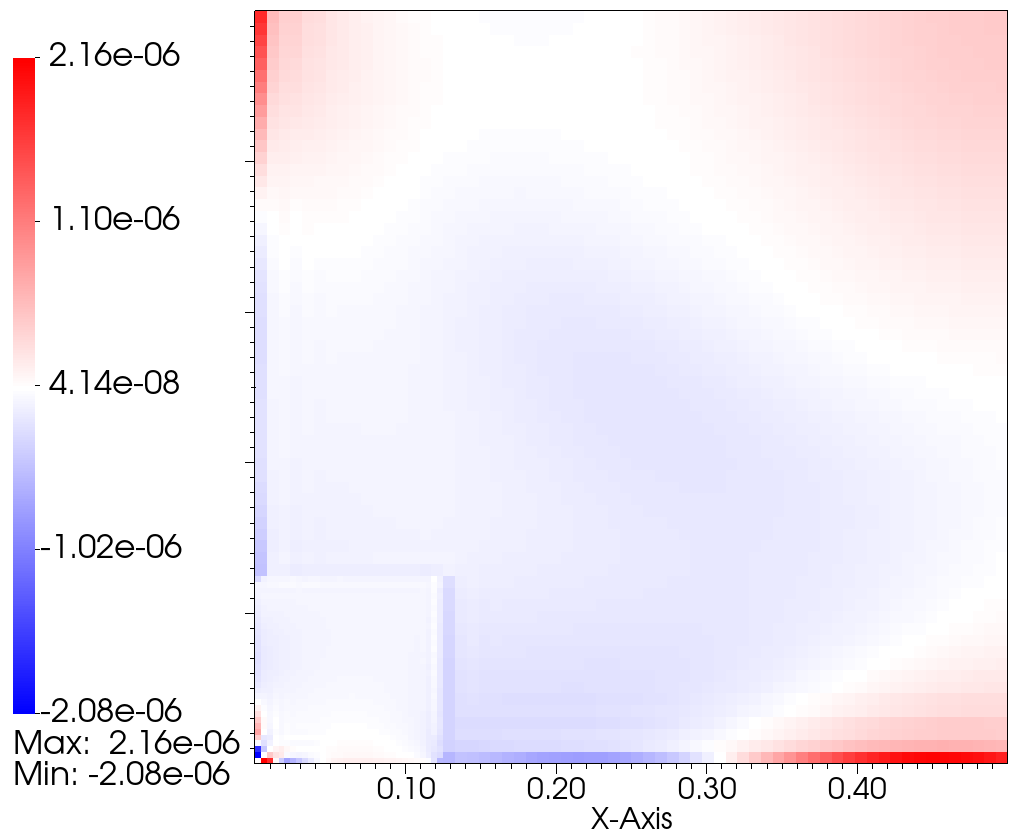}
  }
  \hfill
  \subfigure[$\ell_{\max}=1$ and \mbox{(COH-4)}]
  {
    \includegraphics[scale=0.112]{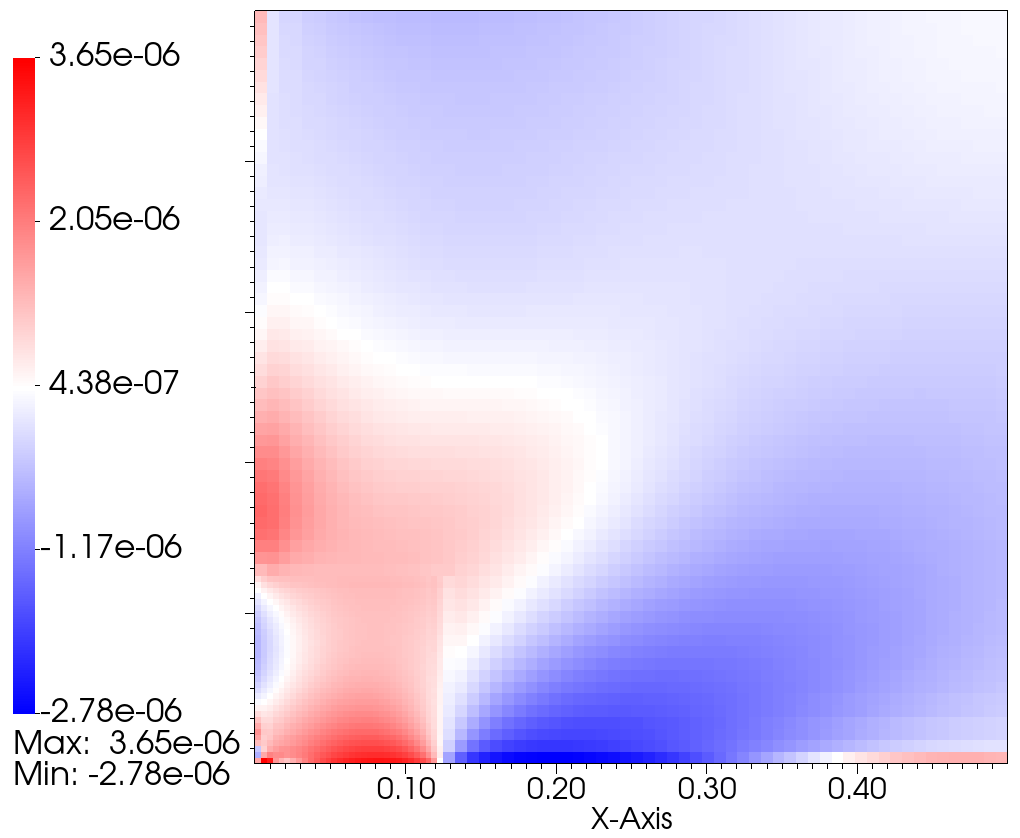}
  }
  \caption{Error snapshots of the horizontal velocity (the first row)
    and the pressure (the second row) for
    the viscous box test in \Cref{tab:viscous-box}
    with $h^0=\frac{1}{128}$. 
    A small box represents the \emph{static} subdomain. 
    ``$\ell_{\max}=0$'', ``no (COH-4)'',  and ``(COH-4)'' 
    correspond to (a), (b), and (c) of \Cref{tab:viscous-box}, 
    respectively. 
  }
  \label{fig:pseudoColor-viscousbox-errors}
\end{figure}

We list errors and convergence rates of the proposed method
 in \Cref{tab:viscous-box}.
Similar to those in \cite[Table 4]{zhang2016gepup}
 and \cite[Section 6.2]{Li2025GePUP-E}, 
 the convergence rates of both velocity and pressure
 in \Cref{tab:viscous-box}(a) 
 based on the $L_{\infty}$ norm are less than 4
 for single-level grids without AMR. 
These order reductions are caused by ${\cal C}^1$ discontinuities
 at the domain corners, 
 where errors in the Stokes pressure, 
 which responds to the commutator of
 the Laplacian and the Leray-Helmholtz projection, 
 may be very large for low Reynolds numbers 
 \cite{cozzi11:_laplac_leray}; 
 see \cite[Section 2.1]{liu2007stability}
 and \cite{zhang2016gepup} for more details.
Figure \ref{fig:pseudoColor-viscousbox-errors}(a,d)
 also show dominant solution errors near the sharp corners.
 
Fortunately,
 these order reductions are substantially alleviated
 by static local refinement.
When the domain corners are covered by the fine level, 
 as shown in Figure \ref{fig:pseudoColor-viscousbox-errors}(b,e),
 dominant errors of both pressure and velocity shift from
 the corners into the interior of the domain.
Consequently, the fourth-order convergence rates of velocity
 in all norms are recovered in Table \ref{tab:viscous-box}(b),
 where we also observe third-order convergence
 for the pressure in terms of the $L_1$ and $L_2$ norms.

The test case of \Cref{tab:viscous-box}(c)
 is the same as that of \Cref{tab:viscous-box}(b)
 except that \mbox{(COH-4)} in \Cref{sec:AMR_composite-operators}
 is always applied to cell-averaged velocities
 after the composite projection. 
It is clear that convergence rates of $\mathbf{u}$
 for all norms are reduced by one. 
Also, error norms of $\mathbf{u}$ on the finest hierarchy
 increase by a factor of at least 15.
In Figure \ref{fig:pseudoColor-viscousbox-errors}(c,f),
 all dominant errors of $\mathbf{u}$ and $p$
 are at the coarse-fine interface,
 with a noticeable jump for the velocity.
These evidences suggest that,
 in synchronizing velocities across multiple levels,
 it is inappropriate
 to replace the projected velocity on invalid regions of a coarse level
 by averages of its counterparts on the finer level.
We speculate that doing so 
 disrupts the smoothness of the velocity on the coarse level, 
 incurring a negative impact upon the accuracy. 
Hereafter, all tests are performed
 without appending \mbox{(COH-4)} 
 to the composite projection.

\subsection{Four-way vortex merging with Re=1000}
\label{sec:four-way-vortex}
Following \cite{Almgren1998ConservativeAdaptiveProjection},
 we use periodic boundary conditions for the domain $\Omega=[0,1]^2$
 and set the initial velocity as 
 $(u_0,v_0) = \left(\frac{\partial \psi}{\partial y},
   -\frac{\partial \psi}{\partial x}\right)$
 where the stream function $\psi$ is the periodic solution
 of Poisson's equation $\Delta \psi = -\omega$
 and the vorticity $\omega$ is the superposition
 of four vortex functions, 
 \begin{displaymath}
   \omega(x,y) = \sum\nolimits_{i=1}^4\frac{1}{2}\Gamma_i
  \left(
     1+\tanh\left[3-100\sqrt{(x-x^o_i)^2+(y-y^o_i)^2}\right]
   \right). 
 \end{displaymath}
For $i=1,2,3,4$,
 the $i$th vortex strength $\Gamma_i$
 is $-150$, $50$, $50$, $50$, 
 and the $i$th vortex center $(x^o_i,y^o_i)$ is
 at $(0.5,0.5)$, $(0.59,0.5)$, $(0.455,0.5+0.045\sqrt{3})$,
 $(0.455,0.5-0.045\sqrt{3})$, respectively.

\begin{table}
  \label{tab:vortexMerging-convergence}
  \centering
  \caption{Error norms and convergence rates
    of the proposed method
    for simulating the four-way vertex merging with Re=1000,
    $t_0=0$, $t_e=0.25$, and $\mathrm{Cr}=0.5$.
    For (c), the convergence rates are estimated with
    a modified Richardson extrapolation \cite[Section 5.4]{zhang2012fourthAMR-ADE}.
  }
   \begin{tabular}{p{2.8cm}|c|c|c|c|c|c}
  \hline \hline
  \multicolumn{2}{c|}{\footnotesize{Grid size $h^{\ell_{\max}}$ of the finest level}}
  & $\frac{1}{256}-\frac{1}{512}$ & rate & $\frac{1}{512}-\frac{1}{1024}$& rate & $\frac{1}{1024}-\frac{1}{2048}$ 
  \\ \hline
  \multirow{6}{75pt}{(a) $\ell_{\max}=0$, i.e., single-level grids}
  & \hspace{4.19pt}$\ubold$ \ $L_{\infty}$ & 5.33e-03 & 3.61 & 4.36e-04 & 3.90 & 2.92e-05
  \\ 
  & $\ubold$ \ $L_1$ & 5.84e-05 & 3.80 & 4.20e-06 & 3.96 & 2.71e-07
  \\ 
  & $\ubold$ \ $L_2$  & 2.56e-04 & 3.72 & 1.94e-05 & 3.93 & 1.27e-06
  \\ \cline{2-7}
  & \hspace{4.19pt}$p$ \ $L_{\infty}$ & 2.12e-03 & 3.68 & 1.65e-04 & 3.96 & 1.06e-05
  \\ 
  & $p$ \ $L_1$ & 4.65e-05 & 3.91 & 3.09e-06 & 3.99 & 1.94e-07
  \\ 
  & $p$ \ $L_2$ & 1.40e-04 & 3.96 & 9.05e-06 & 4.00 & 5.66e-07
  \\ \hline
  \multirow{6}{75pt}{(b) static $\dom^1$ with $\ell_{\max}=1$ and $r=2$}
  &\hspace{4.19pt}$\ubold$ \ $L_{\infty}$  & 5.43e-03 & 3.63 & 4.39e-04 & 3.91 & 2.92e-05
  \\ 
  & $\ubold$ \ $L_1$  & 5.79e-05 & 3.76 & 4.26e-06 & 3.94 & 2.78e-07
  \\ 
  &$\ubold$ \ $L_2$  & 2.50e-04 & 3.69 & 1.94e-05 & 3.93 & 1.27e-06
  \\ \cline{2-7}
  &\hspace{4.19pt}$p$ \ $L_{\infty}$  & 2.10e-03 & 3.69 & 1.63e-04 & 3.94 & 1.06e-05
  \\ 
  &$p$ \ $L_1$  & 4.45e-05 & 3.86 & 3.06e-06 & 3.98 & 1.94e-07
  \\ 
  &$p$ \ $L_2$  & 1.28e-04 & 3.86 & 8.81e-06 & 3.97 & 5.62e-07
  \\ \hline
  \multicolumn{2}{c|}{\footnotesize{Grid size $h^{\ell_{\max}}$ of the finest level}}
  & $\frac{1}{256}-\frac{1}{2048}$ & rate & $\frac{1}{512}-\frac{1}{2048}$& rate & $\frac{1}{1024}-\frac{1}{2048}$
  \\ \hline
  \multirow{6}{75pt}{(c) dynamic $\dom^1$ with $\ell_{\max}=1$, $r=2$, 
  and $|\nabla \times \mathbf{u}| \geq 1$ as the refinement criterion.
  }
  &\hspace{4.19pt}$\ubold$ \ $L_{\infty}$  & 1.54e-02 & 2.95 & 2.37e-03 & 3.31 & 2.17e-04
  \\ 
  & $\ubold$ \ $L_1$ & 3.74e-04 & 3.20 & 4.04e-05 & 3.70 & 2.89e-06
  \\ 
  &$\ubold$ \ $L_2$  & 1.47e-03 & 3.24 & 1.53e-04 & 3.82 & 1.01e-05
  \\ \cline{2-7}
  &\hspace{4.19pt}$p$ \ $L_{\infty}$  & 1.34e-02 & 3.52 & 1.16e-03 & 3.81 & 7.72e-05
  \\ 
  &$p$ \ $L_1$  & 1.89e-04 & 3.40 & 1.78e-05 & 3.68 & 1.29e-06
  \\ 
  &$p$ \ $L_2$  & 8.17e-04 & 3.41 & 7.60e-05 & 3.78 & 5.16e-06
  \\ \hline \hline
\end{tabular}

\end{table}

\begin{figure}
  \centering
    \subfigure[$t=0.05$]
  {
    \includegraphics[scale=0.165]{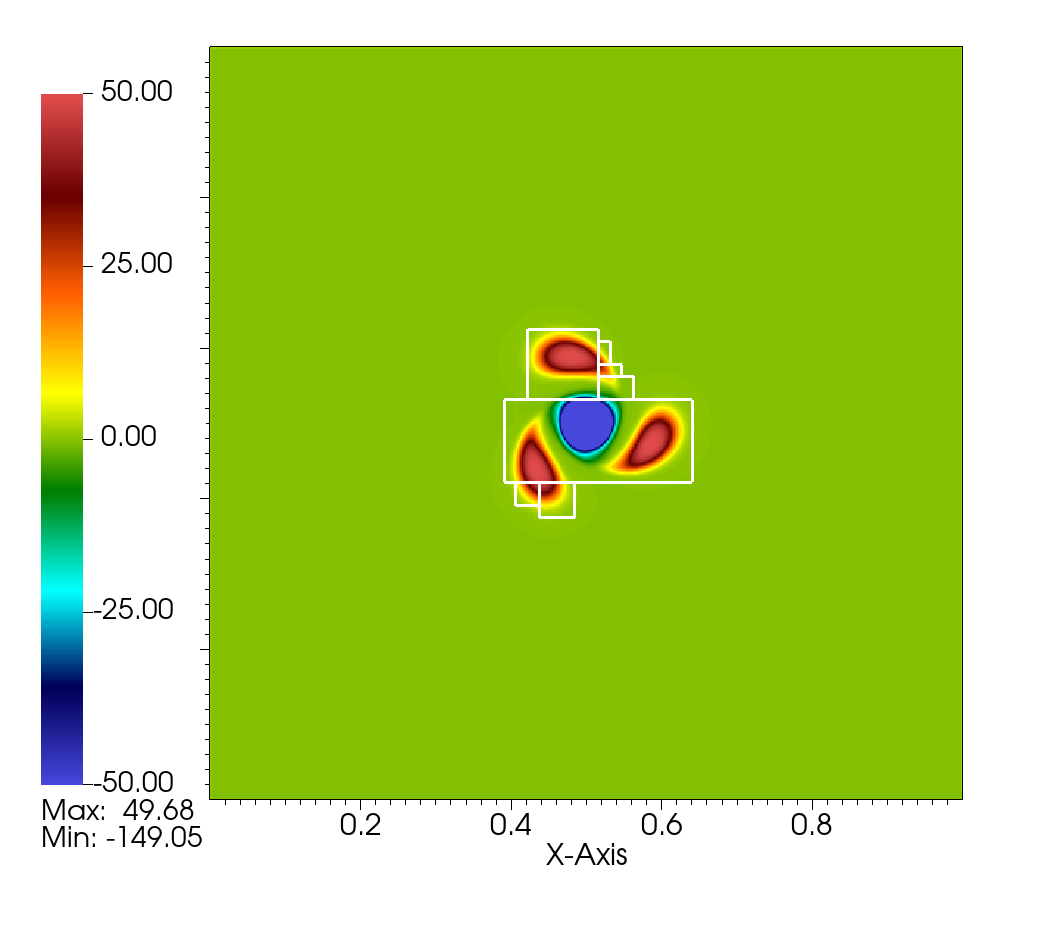}
  }
  \hfill
  \subfigure[$t=0.25$]
  {
    \includegraphics[scale=0.165]{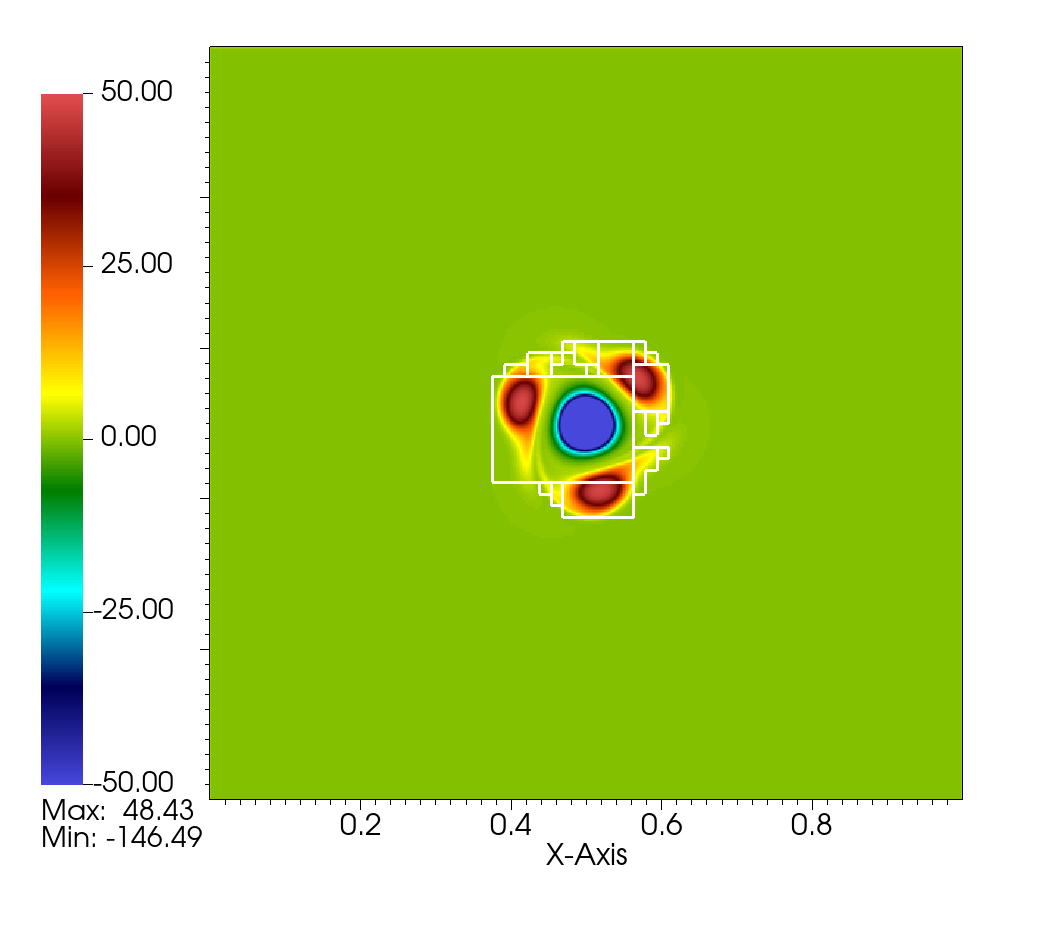}
  }
  \caption{Vorticity snapshots of the four-way vortex merging test
    for case (c) of \Cref{tab:vortexMerging-convergence}
    with $h^0 = \frac{1}{512}$, $r=2$, and $\mathrm{Cr}=0.5$.
    The refinement level is represented by white boxes.
  }
  \label{fig:vortexMerging-AMR-vorticitySnapshots}
\end{figure}

For Re=1000, 
 we advance the cell-averaged initial velocity from $t_0=0$ to $t_e=0.25$
 on single-level grids,
 statically refined grids with $\dom^1=[0.25,0.75]^2$, 
 and dynamically refined grids
 with $|\nabla \times \mathbf{u}| \geq 1$
 as the refinement criterion. 
As shown in Table \ref{tab:vortexMerging-convergence}(a,b),
 the proposed method achieves fourth-order convergence rates
 for $\mathbf{u}$ and $p$ in all norms
 on single-level grids and statically refined grids.
The error norms on statically refined grids
 are very close to those on single-level grids,
 indicating that the solution errors on $\Omega\setminus\dom^1$
 have already been reduced to a negligible level by the coarsest grid.

In Table \ref{tab:vortexMerging-convergence}(c), 
 the convergence rates of $\mathbf{u}$ and $p$ in all norms
 (except the $L_{\infty}$ norm for velocity) 
 on dynamically refined grids
 are close to 4, 
 which implies that the order reduction in the $L_{\infty}$-norm of velocity
 only happens at an $O(1)$ number of locations.
By the paragraph below (\ref{eq:compositeErrorDynamic}),
 the implication of this order reduction is not
 that the proposed method fails to be fourth-order accurate 
 but that the criterion of dynamic refinement
 has missed dominating errors at an $O(1)$ number of locations. 
Consequently, as shown in the last column of \Cref{tab:vortexMerging-convergence}, 
 the solution errors on dynamically refined grids
 are roughly ten times larger than those on statically refined grids.
On the other hand,
 the area of dynamically refined regions
 shown in Figure \ref{fig:vortexMerging-AMR-vorticitySnapshots}
 is much smaller than that of the statically refined region $\dom^1=[0.25,0.75]^2$.
This discussion illustrates the flexibility of dynamic grid refinement
 in balancing accuracy and efficiency.

\subsection{Single-vortex test with Re=20,000}
\label{sec:single-vortex-test}
On the unit box $\Omega=[0,1]^2$, %
 we first define an axisymmetric velocity field
 \begin{equation*}
   \label{eq:singleVortexBox}
   \renewcommand{\arraystretch}{1.2}
   u_{\theta}(r_v) =
   \left\{
     \begin{array}{cc}
       \Gamma(\frac{1}{2}r_v - 4r_v^3) & \text{ if } r_v<R ;
       \\
       \Gamma\frac{R}{r_v}(\frac{1}{2}R - 4R^3) & \text{ otherwise}, 
     \end{array}
   \right.
 \end{equation*}
 where $r_v$ is the distance from the domain center $(\frac{1}{2}, \frac{1}{2})^{\top}$
 and $R=0.2$ and $\Gamma=1$
 give $U^*:=\max(u_{\theta}) = 0.068$.
Then we project ten times cell averages of ${u}_{\theta}$
 onto the divergence-free space with no-slip boundary conditions 
 to obtain the cell-averaged initial velocity.
A small kinematic viscosity $\nu^* = 3.4\times 10^{-6}$
 gives $\mathrm{Re}=20,000$.

 \begin{figure}
  \centering
  \subfigure[$t=40$]
  {
    \includegraphics[scale=0.18]{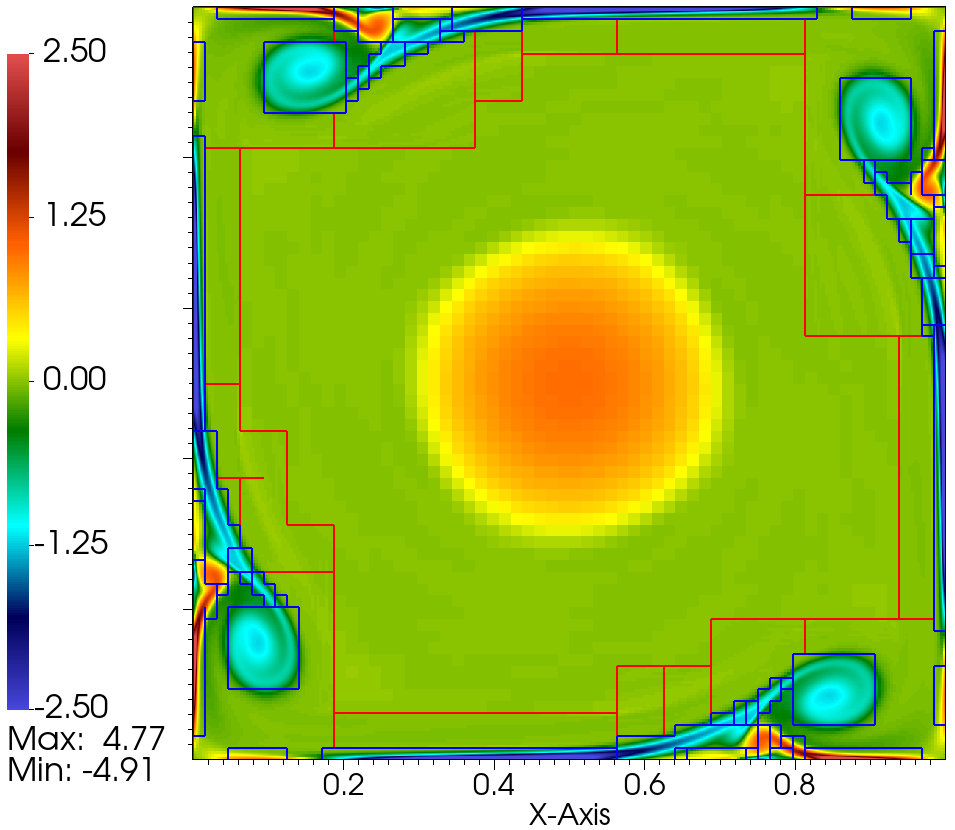}
  }
  \subfigure[$t=60$]
  {
    \includegraphics[scale=0.18]{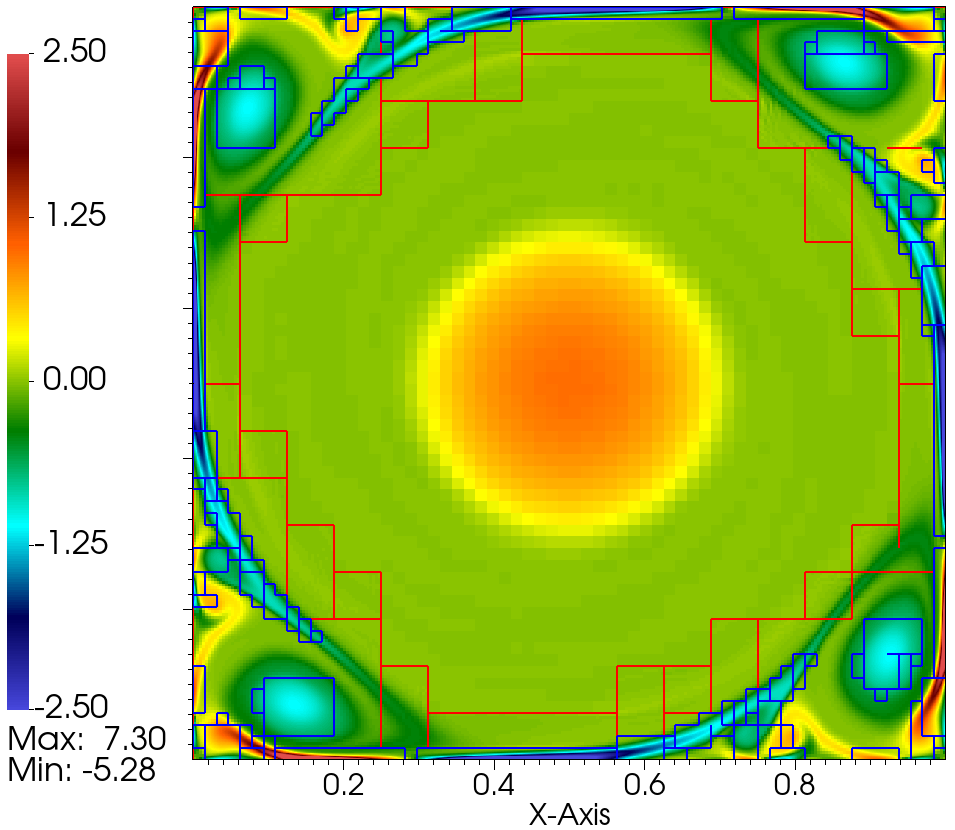}
  }
  \caption{Vorticity snapshots of the single-vortex test
    on a three-level AMR hierarchy with
    \mbox{$h^0 = \frac{1}{64}$}, $r=4$, and $\mathrm{Cr}=0.5$.
    Red boxes and blue boxes respectively represent $\Omega^1$ and
    $\Omega^2$, which are generated by the refinement criteria
    $|\nabla \times \mathbf{u} |\geq 0.15$
    and 
    $|\nabla \times \mathbf{u} |\geq 0.8$, respectively.
    }
  \label{fig:single-vortex-vorticitySnapshots}
\end{figure}

\begin{figure}
  \centering
  \subfigure[(TLS-1) for single vortex]
  {  
    \includegraphics[scale=0.273]{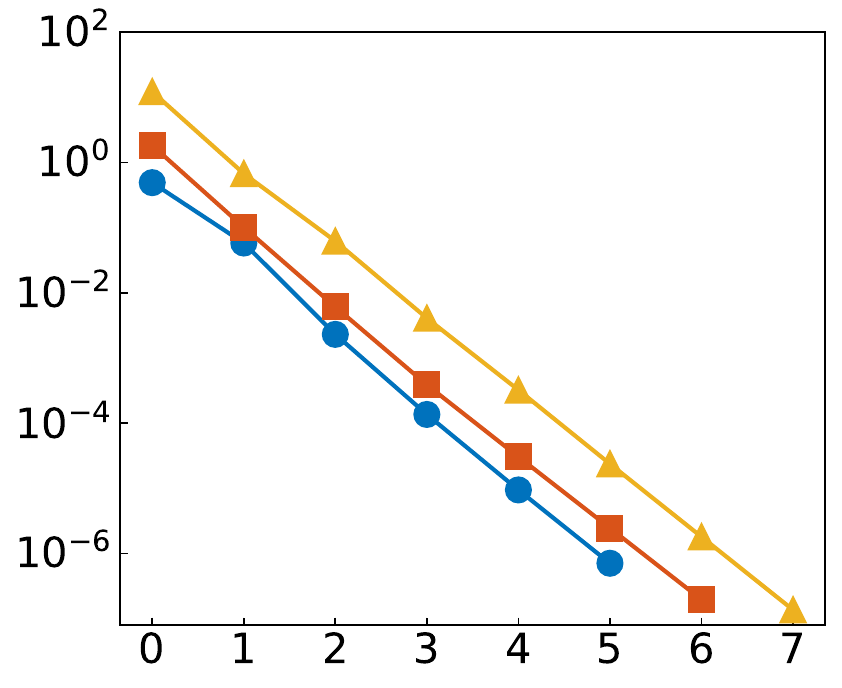}
  }
  \hfill
  \subfigure[(TLS-2) for single vortex]
  {  
    \includegraphics[scale=0.273]{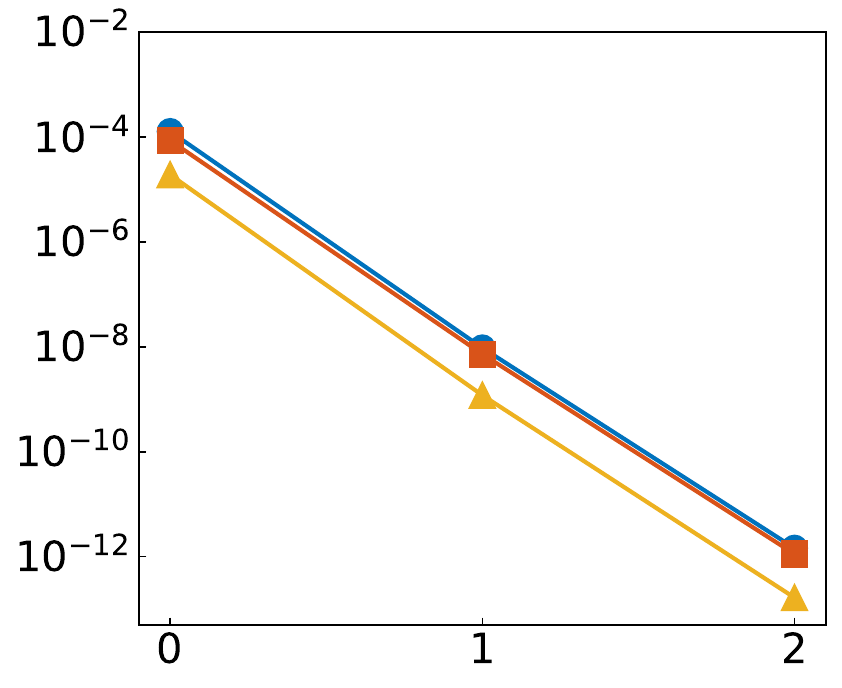}
  }
  \hfill
  \subfigure[(TLS-3) for single vortex]
  {  
    \includegraphics[scale=0.273]{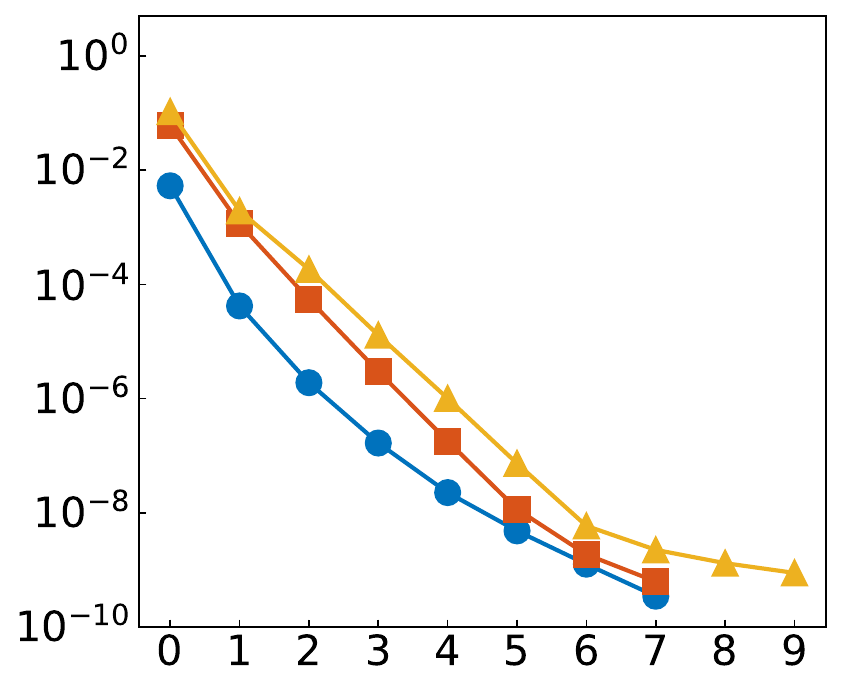}
  }
  \\
  \subfigure[(TLS-1) for dipole collison]
  { 
    \includegraphics[scale=0.273]{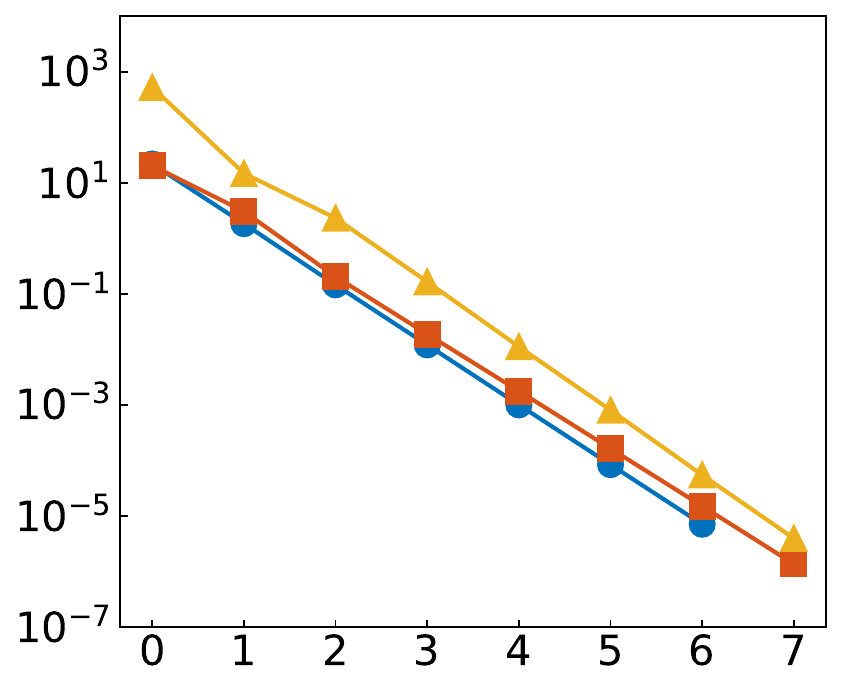}
  }
  \hfill
  \subfigure[(TLS-2) for dipole collison]
  {  
    \includegraphics[scale=0.273]{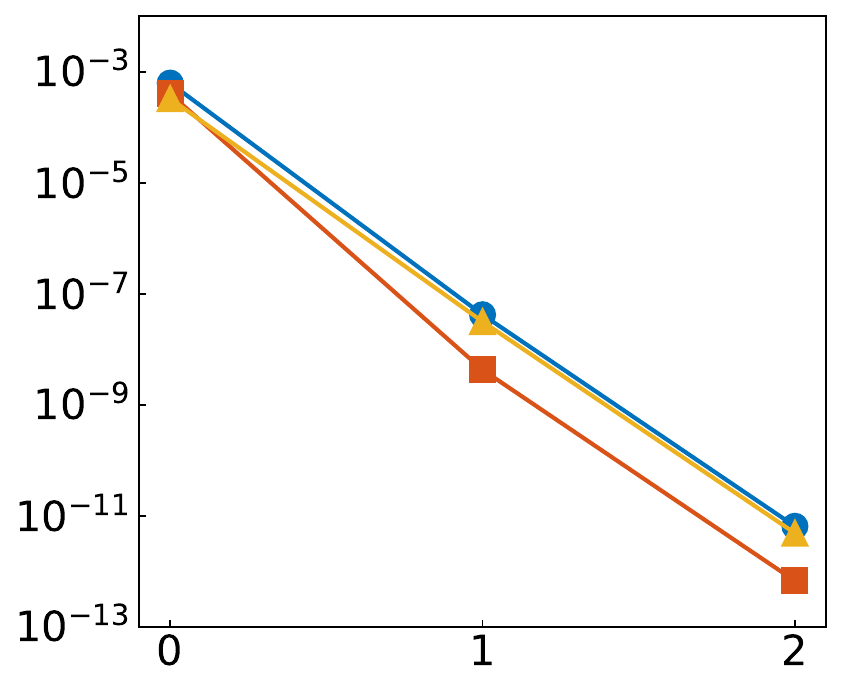}
  }
  \hfill
  \subfigure[(TLS-3) for dipole collison]
  {  
    \includegraphics[scale=0.273]{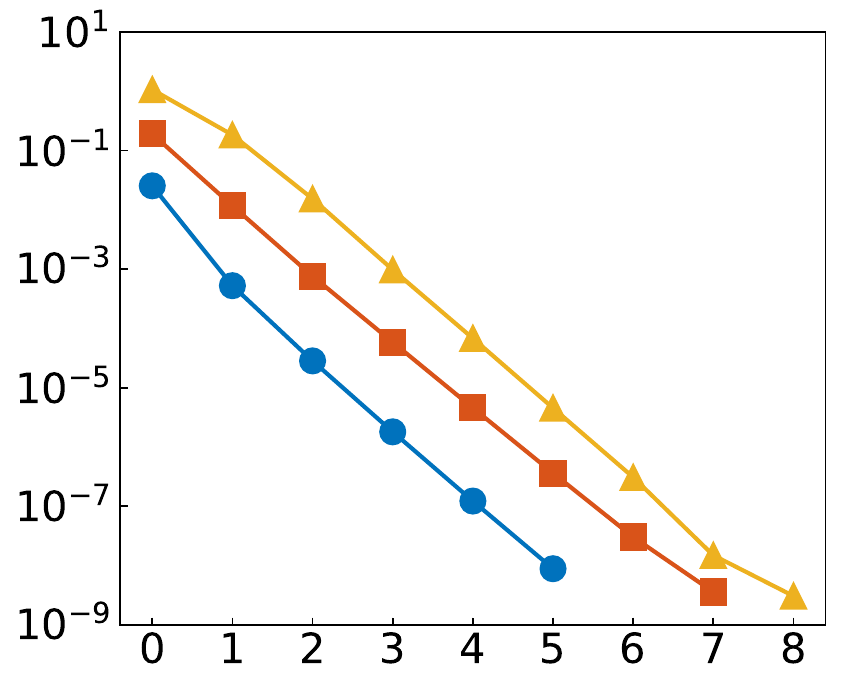}
  }
  \caption{Performance of multigrid V-cycles in
      Algorithm \ref{alg:time-stepping_procedure} for
      solving the single-vortex test 
      and the dipole vortex-wall collision test 
      at the final time step.
     The abscissa and the ordinate are
      the iteration number of multigrid V-cycles
      and the max-norm of the residual of the solution, respectively.
      The lines marked with  ``$\bullet$, ``{\tiny $\blacksquare$}'',
      and ``$\blacktriangle$''
      represent the residuals for level 0, 1, and 2, respectively.
      In all tests, 
      the numbers of pre-smoothing and post-smoothing are both set to $4$,
      and the relative and absolute residual convergence thresholds are set to
      $10^{-8}$ and  $10^{-12}$, respectively.
    }
  \label{fig:mg_convergenceRate}    
\end{figure}

The initial velocity is advanced from $t_0=0$ to $t_e=60$
 by the proposed method on a three-level AMR hierarchy with $r=4$
 and dynamic regridding,
 with the vorticity field plotted at two key time instances
 in \Cref{fig:single-vortex-vorticitySnapshots}. 
The mosaic pattern on the coarsest level is significantly reduced
 on the intermediate level 
 and becomes indiscernible on the finest level.
The prominent features of the vortex sheet roll-up
 and the formation of counter-vortices agree
 with those in \cite{bell1991efficient,zhang2016gepup}.
In particular,
 the coherent structures shown in \Cref{fig:single-vortex-vorticitySnapshots}
 are visually indistinguishable
 from those on single-level grids in \mbox{\cite[Fig. 4]{zhang2016gepup}}.
The performance of the geometric multigrid method
 in solving the three linear systems (TLS-1,2,3) 
 is shown in Figure \ref{fig:mg_convergenceRate}(a-c),
 where the residual reduction rate
 is around 10 for (TLS-1,3)
 and about $10^4$ for (TLS-2).
 
\begin{figure}
  \centering
  \subfigure[$t=5$]
  {
    \includegraphics[width=6cm, height=2.7cm, keepaspectratio=false]{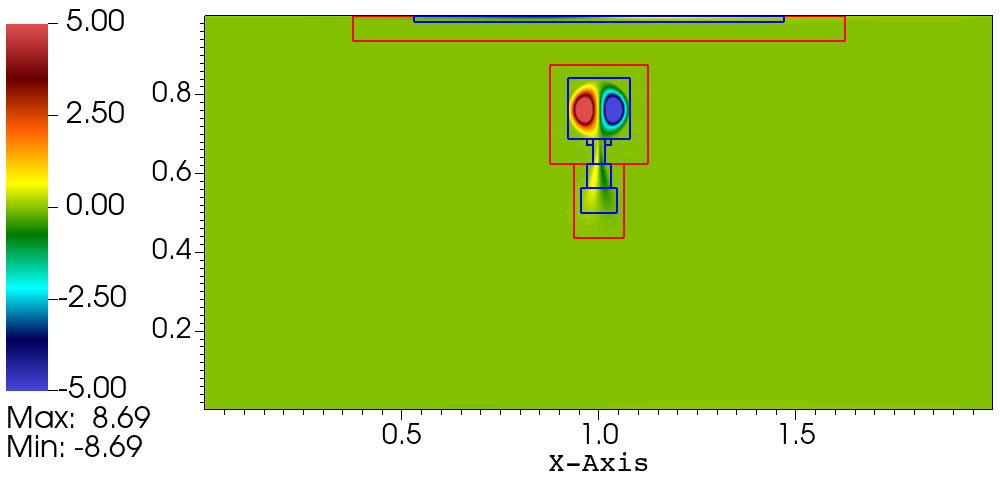}
  }
  \subfigure[$t=10$]
  {
    \includegraphics[width=6cm, height=2.7cm, keepaspectratio=false]{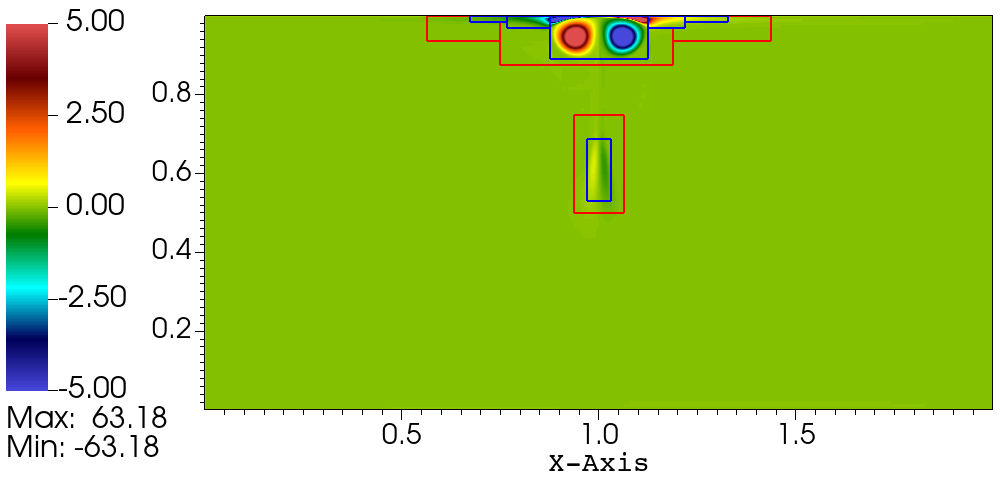}
  }
  \subfigure[$t=15$]
  {
    \includegraphics[width=6cm, height=2.7cm, keepaspectratio=false]{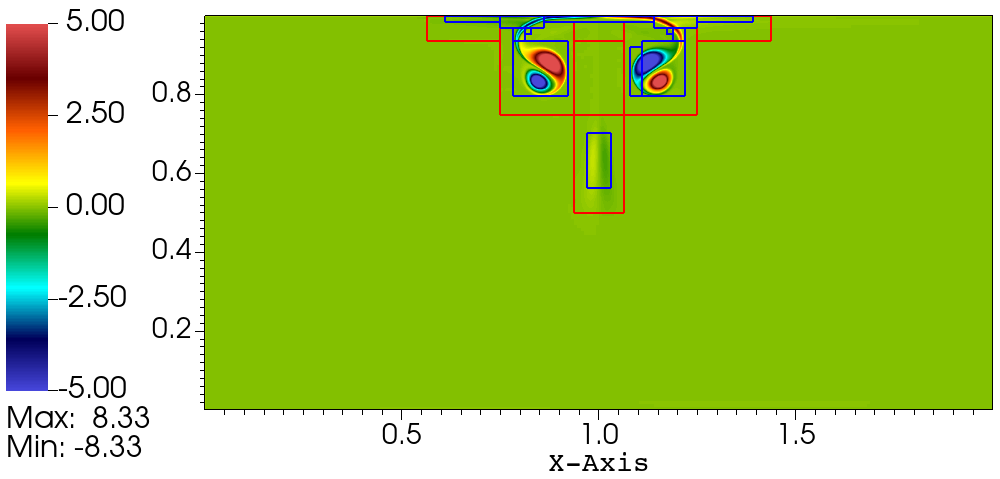}
  }
  \subfigure[$t=20$]
  {
    \includegraphics[width=6cm, height=2.7cm, keepaspectratio=false]{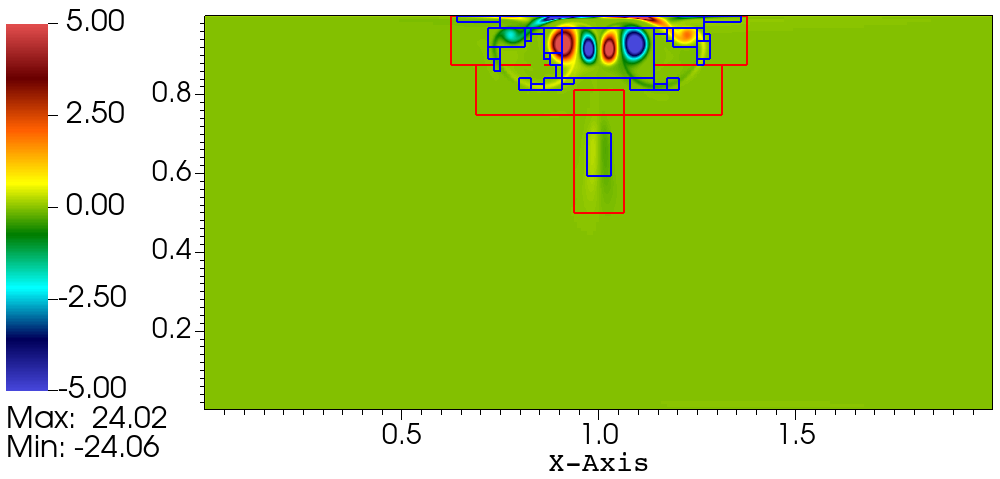}
  }
  \caption{Vorticity snapshots of the dipole vortex-wall collision
    test on a three-level AMR hierarchy with $h^0 = \frac{1}{128}$, 
    $r=4$, and $\mathrm{Cr}=0.5$.
    Red boxes and blue boxes respectively represent $\Omega^1$ and
    $\Omega^2$, which are generated by the refinement criteria
    $|\nabla \times \mathbf{u} |\geq 0.05$
    and 
    $|\nabla \times \mathbf{u} |\geq 0.2$, respectively.
  }
  \label{fig:Dipole-vorticitySnapshots}
\end{figure}

\subsection{Dipole vortex-wall collision with Re=42,000}
\label{sec:dipole-vortex-wall}

The rectangular domain $\Omega = [0,2]\times[0,1]$
 of this test has periodic boundaries at $x=0,2$
 and no-slip walls at $y=0,1$.
Via the same process introduced in \Cref{sec:four-way-vortex}, 
 we deduce the initial velocity 
 from the vorticity 
 \begin{equation}
   \label{eq:dipoleVortex}
   \omega(x,y) = - \omega_e x \exp\left[-\frac{(x-x_0)^2 + (y-y_0)^2}{\sigma^2}\right],
 \end{equation}
 where $\omega_e=600$ is the vortex strength,
 $(x_0,y_0)=(1,0.5)$ the center of the dipolar vorticity,
 and $\sigma = 0.0375$ the size of the vortical structure.
These values lead to the maximum velocity at about $0.21$, 
 which, together with $\nu^*\!=\!5\!\times\!10^{-6}$,
 gives $\mathrm{Re} = 42,000$. 

Cell averages of the initial velocity is advanced
 from $t_0=0$ to $t_e=20$
 by the proposed method on a three-level AMR hierarchy with $r=4$
 and dynamic regridding,
 with the vorticity field at key instances
 plotted in \Cref{fig:Dipole-vorticitySnapshots}.
The dipolar vortices move upwards,
 collide with the top boundary at $y=1$,
 form a viscous boundary layer,
 rebound from the wall, 
 and roll up into secondary vortices.
These prominent features are captured
 by dynamically generated grids shown in \Cref{fig:Dipole-vorticitySnapshots}; 
 our results agree well with those in \cite{kramer2007vorticity,van2022fourth}.
As shown in Figure  \ref{fig:mg_convergenceRate}(d-f),
 the performance of the geometric multigrid method
 for solving this test
 is very similar to that
 for the single-vortex test in Figure
 \ref{fig:mg_convergenceRate}(a-c),
 once again confirming the efficiency of the geometric multigrid method
 in \Cref{subsec:line-solv-irreg}.

\subsection{Efficiency evaluation}
 \label{subsec:alg:effic-eval}

The superior efficiency of AMR to uniform single-level grids
 is usually demonstrated by the large ratio
 of the CPU time consumed on a uniform single-level grid 
 to that on an AMR hierarchy
 for completing the \emph{entire} test. 
In this work, we delve deeper into
 the temporal variation of locally refined regions
 by examining two \emph{time-dependent} ratios that measure
 the ideal and actual speedup of AMR over uniform grids.
We start with two assumptions as follows.
 
\begin{enumerate}[({EVA}-1)]
\item The CPU time for solving linear systems dominates
  that of all other modules 
  such as regridding, evaluating discrete operators, 
  and enforcing coarse-fine interface conditions and physical boundary conditions. 
\item The complexity of solving linear systems is \emph{optimal},
  i.e., the corresponding CPU time is linearly proportional to the number of unknowns.
\end{enumerate}
Results of extensive numerical tests of the proposed method
 confirm (EVA-2).

It follows that,
 to advance the numerical solution from $t^n$ to $t^{n+1}=t^n+k^0$, 
 the CPU times consumed by subcycled AMR 
 and by a single-level grid with uniform spacing $h^{\ell_{\max}}$
 are, respectively, 
 \begin{displaymath}
   R^n_{\mathrm{amr}}
   :=
   \sum^{\ell_{\max}}_{\ell=0}\frac{k^0}{k^{\ell}}
   \Theta(N_{\ell})
   \quad\mathrm{and}\quad
   R^n_{\mathrm{unf}} :=
   \frac{k^0}{k^{\ell_{\max}}}\Theta(N_{\mathrm{unf}}),
 \end{displaymath}
 where $N_{\ell}$ is the number of cells in $\Omega^{\ell}$, 
 $N_{\mathrm{unf}}$ the number of cells in the single-level grid,
 $\Theta(N_{\ell})$ the optimal complexity of subcycled AMR 
 to march one time step with size $k^\ell$,
 and $\Theta(N_{\mathrm{unf}})$ that of the single-level grid
 to march one time step with size $k^{\ell_{\max}}$.

 The \emph{ideal speedup}
 of the $n$th time step is defined as
 \begin{equation}
   \label{eq:ideal_speedup}
   S^n_{\mathrm{idl}} := \frac{R^n_{\mathrm{unf}}}{R^n_{\mathrm{amr}}}
   \approx \frac{1}{\sum^{\ell_{\max}}_{\ell=0}
     \frac{k^{\ell_{\max}}}{k^{\ell}}\frac{N_\ell}{N_{\mathrm{unf}}}},
 \end{equation}
 where the last step follows from the assumption
 that $\Theta(N_{\ell})$ and
 $\Theta(N_{\mathrm{unf}})$ have roughly the same constant.
The value of $S^n_{\mathrm{idl}}$ is completely determined
 by the regridding steps (RDM-1,2,3,4) in \Cref{subsec:regridding}. 

The \emph{actual speedup}s of one time step
 and of the entire simulation
 are respectively 
 \begin{equation}
   \label{eq:actualSpeedups}
   S^n_{\mathrm{act}} :=
   \frac{T^n_{\mathrm{unf}}}{T^n_{\mathrm{amr}}};
   \quad
   S_{\mathrm{act}} :=
   \frac{\sum_n T^n_{\mathrm{unf}}}{\sum_n T^n_{\mathrm{amr}}}, 
 \end{equation}
 where $T^n_{\mathrm{amr}}$ and $T^n_{\mathrm{unf}}$ denote
 the CPU time consumed at the $n$th time step
 on the AMR hierarchy and the uniform grid, respectively.
Different from $S^n_{\mathrm{idl}}$,
 the actual speedup $S^n_{\mathrm{act}}$ is affected by
 all algorithmic and implementational aspects of subcycled AMR.
The ratio $\frac{S^n_{\mathrm{act}}}{S^n_{\mathrm{idl}}}$ being close
 to one indicates that \mbox{(EVA-1)} is a good assumption for this time step.
$S_{\mathrm{act}}$ is the common ratio
 for measuring the efficiency of AMR mentioned
 in the opening paragraph of this subsection.

\begin{figure}
  \centering
  \subfigure[$S_{\mathrm{act}}\approx 9$ for the single-vortex test]
  {  \label{fig:single-vortex-efficiency-curve}  
    \includegraphics[scale=0.295]{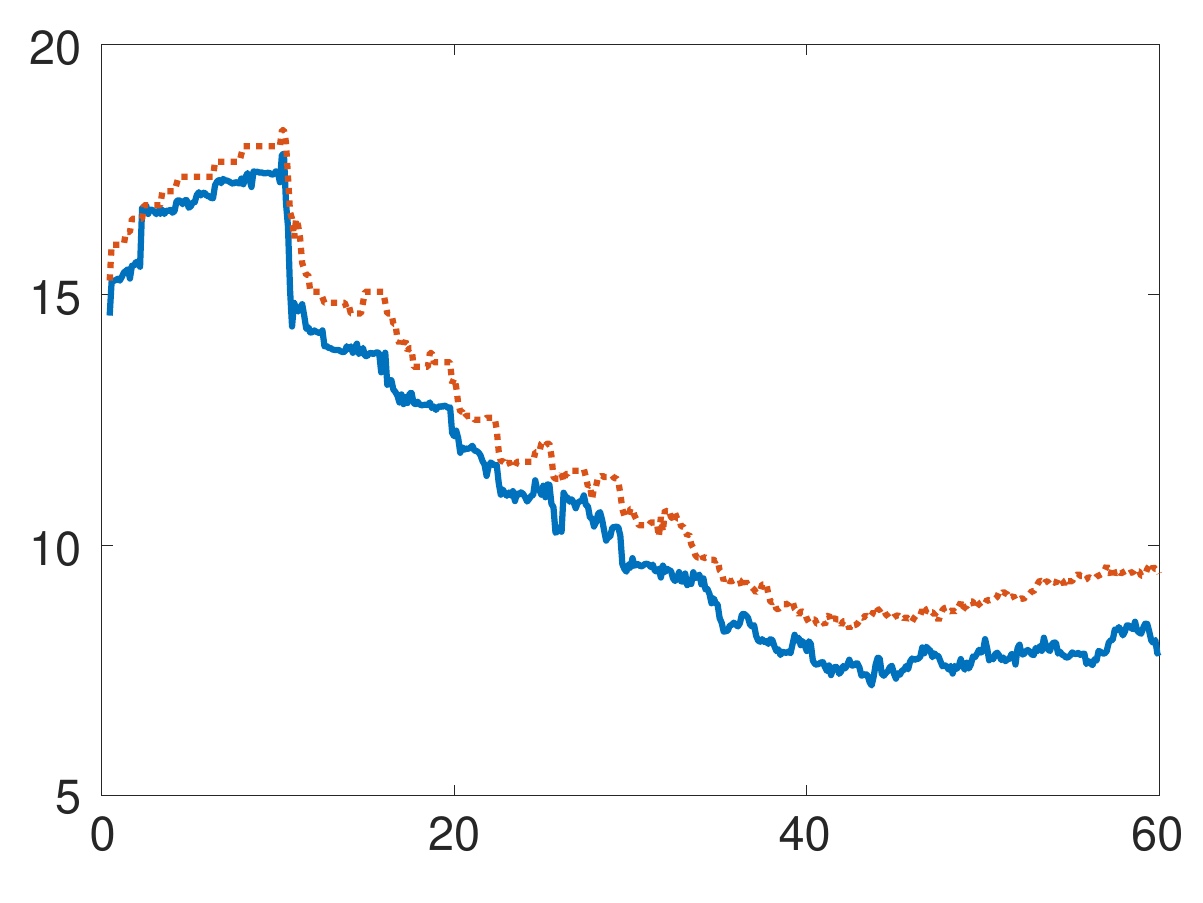}
  }
  \hfill
  \subfigure[$S_{\mathrm{act}}\approx 29$ for the dipole vortex-wall collision]
  {  \label{fig:Dipole-efficiency-curve}
    \includegraphics[scale=0.295]{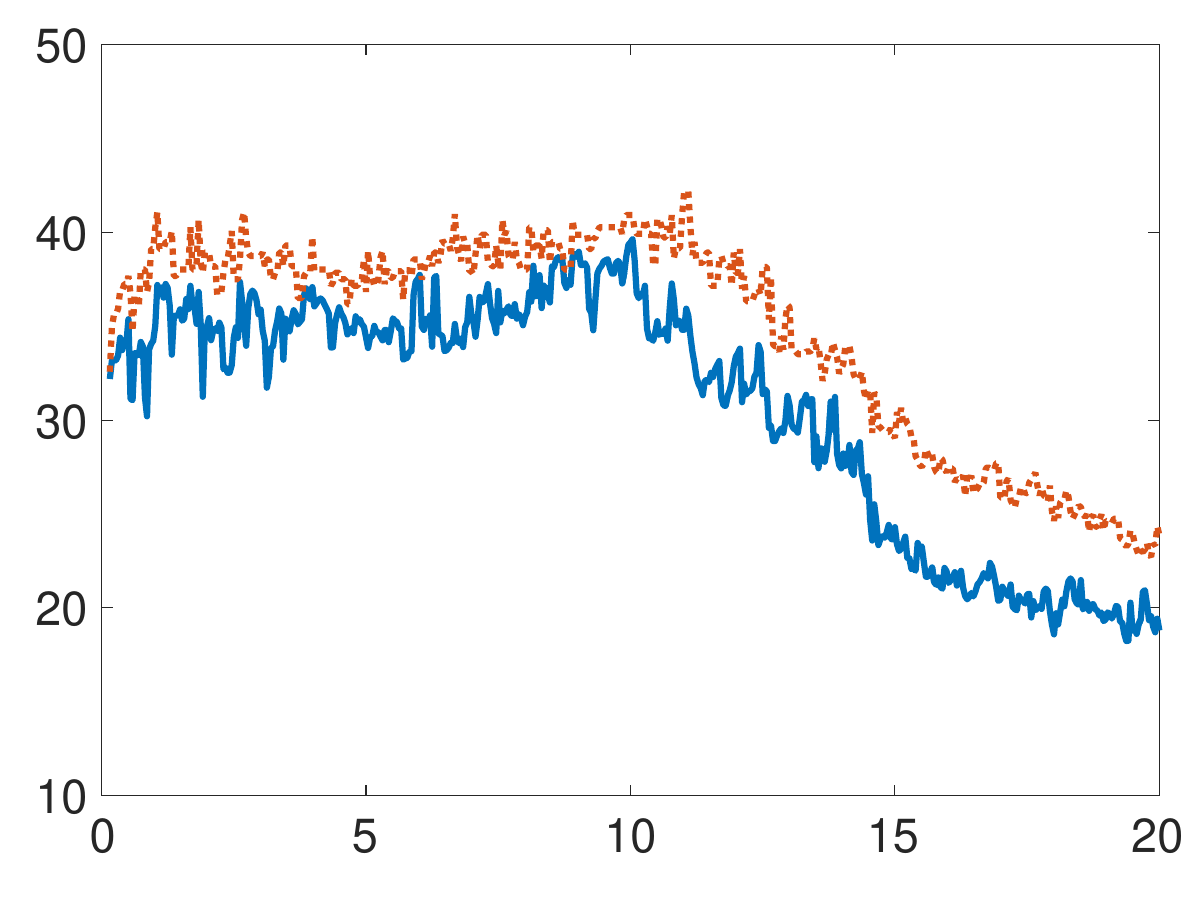}
  }
  \caption{Temporal evolutions of
    the ideal speedup  $S^n_{\mathrm{idl}}$ (orange dotted lines)
    and the actual speedup $S^n_{\mathrm{act}}$ (blue solid lines)
    for the tests in \Cref{sec:single-vortex-test}
    and \Cref{sec:dipole-vortex-wall},
    where the uniform spacing of the single-level grid
    is $h^{\ell_{\max}}$, i.e., 
    the grid size of the finest level of subcycled AMR.
    The horizontal and vertical axes represent time
    and speedup ratio of subcycled AMR, respectively.
    The maximum number of cells within the entire AMR hierarchy
    is much smaller than 
    the number of cells on the single-level grid
    by a factor of 11 and 30 for (a) and (b), respectively.
  }
  \label{fig:speedupEvolutions}    
\end{figure}

In \Cref{fig:speedupEvolutions},  we present
 temporal evolutions of the speedup ratios
 for the two tests in Subsections \ref{sec:single-vortex-test}
 and \ref{sec:dipole-vortex-wall}.
Both subplots suggest the following.

\begin{itemize}
\item Values of the actual speedup $S^n_{\mathrm{act}}$
 are always close to those of the ideal speedup $S^n_{\mathrm{idl}}$,
 confirming (EVA-1). 
\item Both evolutions can be roughly divided into two stages by a key instant $t_*$: 
 $S^n_{\mathrm{idl}}$ remains more or less constant during $[t_0, t_*]$
 and mostly decreases in time during $[t_*,t_e]$, 
 due to the fact that regions of high vorticity
 are concentrated in a small region during $[t_0, t_*]$
 and become more spread out during $[t_*,t_e]$. 
For both \Cref{fig:single-vortex-efficiency-curve}
 and \Cref{fig:Dipole-efficiency-curve},
 we have $t_*\approx 10$. 
\item The curve of $S^n_{\mathrm{act}}$
  is closer to that of $S^n_{\mathrm{idl}}$
  in the first stage than in the second stage,
  because the more concentrated regions of high vorticity
  lead to a simpler topology and geometry of the refinement levels, 
  which further imply a smaller percentage of the overhead 
  in managing the coarse-fine interface.
\end{itemize}

Finally, we note in passing that
 values of speedup ratios in
 (\ref{eq:ideal_speedup}) and (\ref{eq:actualSpeedups})
 depend largely 
 on the problem at hand, the user-specified refinement criteria,
 and the number of refinement levels.



\section{Conclusion}
\label{sec:conclusion}

We have developed a fourth-order adaptive projection method
 for solving the INSE with subcycling in time.
To enforce the divergence-free constraint, 
 we adapt the GePUP-E formulation \cite{Li2025GePUP-E}
 of INSE in the context of AMR
 and derive coarse-fine interface conditions
 so that the velocity divergence decays exponentially
 on the subdomain of any refinement level. 
For subcycling in time, 
 we recursively advance the velocity on a single level
 and its finer levels, 
 with the interface conditions approximated 
 via spatiotemporal interpolations 
 and with the solvability conditions of elliptic equations
 satisfied to machine precision
 for each connected component of the subdomain.
Within each time step,
 the algorithm mainly consists
 of solving a sequence of linear systems by geometric multigrid, 
 leading to an optimal complexity of the proposed method.
By adopting implicit-explicit schemes for time integration, 
 the proposed method is also applicable
 to a wide range of Reynolds numbers.
Results of numerical tests
 confirm the fourth-order accuracy of velocity in the $L_{\infty}$ norm
 and the third-order accuracy of pressure in the $L_2$ norm.
The superior efficiency of our subcycled AMR method
 is also demonstrated by temporal variations of two speedup ratios 
 over uniform grids.
Not restricted to adaptive grids made by rectangular patches,
 the main components of the proposed AMR methods
 can also be transplanted to adaptive grids based on quadtrees or
 octrees.

The next step along this research line
 is to augment the subcycled AMR 
 to solve INSE on irregular domains
 via poised lattice generation \cite{zhang2024PLG}.
The extension of our method to three dimensions 
 should be theoretically straightforward, 
 but may involve practical difficulties
 of designing efficient multigrid solvers,
 especially on irregular domains.

Finally,
 we plan to couple the proposed AMR method
 with our fourth-order interface tracking methods
 \cite{hu24:_arms,TaQi25}
 to simulate incompressible viscous fluids
 with moving boundaries. 



\section*{Acknowledgments}

We thanks two anonymous referees 
  for their comments and suggestions
  that lead to an improvement on the exposition.
  We also acknowledge one helpful suggestion from Jiatu Yan, 
  a graduate student in our team
  at the School of Mathematical Sciences, Zhejiang University.

\bibliographystyle{siamplain}
\bibliography{GePUP-AMR.bib}
\end{document}